\documentclass[reqno,11pt]{amsart}
\usepackage{amsmath}
\usepackage{amsxtra}
\usepackage{amscd}
\usepackage{amsthm}
\usepackage{amsfonts}
\usepackage{amssymb}
\usepackage{eucal}
\usepackage{scalerel}
\usepackage{verbatim}
\usepackage[matrix,arrow,curve]{xy}

\usepackage{a4wide}

\usepackage[T1]{fontenc}

\usepackage{xr-hyper}
\usepackage[
pdftex,
bookmarks=false,
colorlinks=true,
debug=true,
pdfnewwindow=true]{hyperref}

\theoremstyle{plain}
\newtheorem{Thm}[equation]{Theorem}
\newtheorem{Cor}[equation]{Corollary}
\newtheorem{Lem}[equation]{Lemma}
\newtheorem{Prop}[equation]{Proposition}
\newtheorem{Conj}[equation]{Conjecture}

\theoremstyle{definition}
\newtheorem{Def}[equation]{Definition}

\theoremstyle{remark}

\newtheorem{Rem}[equation]{Remark}

\newtheorem{conj}{Conjecture}

\numberwithin{equation}{section}
\renewcommand{\rm}{\normalshape}

\newif\ifShowLabels
\ShowLabelstrue
\newdimen\theight
\def\TeXref#1{%
    \leavevmode\vadjust{\setbox0=\hbox{{\tt
        \quad\quad  {\small \rm #1}}}%
    \theight=\ht0
    \advance\theight by \lineskip
    \kern -\theight \vbox to
    \theight{\rightline{\rlap{\box0}}%
    \vss}%
    }}%

\ShowLabelsfalse

\newenvironment{thm}[1]%
    { \begin{Thm} \label{T:#1}  \ifShowLabels \TeXref{T:#1} \fi }%
    { \end{Thm} }

\renewcommand{\th}[1]{\begin{thm}{#1} \sl }
\renewcommand{\eth}{\end{thm} }

\newenvironment{lemma}[1]%
    { \begin{Lem} \label{L:#1}  \ifShowLabels \TeXref{L:#1} \fi }%
    { \end{Lem} }
\newcommand{\lem}[1]{\begin{lemma}{#1} \sl}
\newcommand{\elem}{\end{lemma}}

\newenvironment{propos}[1]%
    { \begin{Prop} \label{P:#1}  \ifShowLabels \TeXref{P:#1} \fi }%
    { \end{Prop} }
\newcommand{\prop}[1]{\begin{propos}{#1}\sl }
\newcommand{\eprop}{\end{propos}}

\newenvironment{corol}[1]%
    { \begin{Cor} \label{C:#1}  \ifShowLabels \TeXref{C:#1} \fi }%
    { \end{Cor} }
\newcommand{\cor}[1]{\begin{corol}{#1} \sl }
\newcommand{\ecor}{\end{corol}}

\newenvironment{defeni}[1]%
    { \begin{Def} \label{D:#1}  \ifShowLabels \TeXref{D:#1} \fi }%
    { \end{Def} }
\newcommand{\defe}[1]{\begin{defeni}{#1} \sl }
\newcommand{\edefe}{\end{defeni}}

\newenvironment{remark}[1]%
    { \begin{Rem} \label{R:#1}  \ifShowLabels \TeXref{R:#1} \fi }%
    { \end{Rem} }
\newcommand{\rem}[1]{\begin{remark}{#1}}
\newcommand{\erem}{\end{remark}}

\newenvironment{conjec}[1]%
    { \begin{Conj} \label{Co:#1}  \ifShowLabels \TeXref{Co:#1} \fi }%
    { \end{Conj} }
\renewcommand{\conj}[1]{\begin{conjec}{#1} \sl }
\newcommand{\econj}{\end{conjec}}

\newcommand{\eq}[1]%
    { \ifShowLabels \TeXref{E:#1} \fi
       \begin{equation} \label{E:#1} }
\newcommand{\eeq}{ \end{equation} }

\newcommand{\prf}{ \begin{proof} }
\newcommand{\epr}{ \end{proof} }

\usepackage[dvipsnames]{xcolor}

\newcommand\nc{\newcommand}
\nc{\unl}{\underline}
\nc{\ol}{\overline}
\nc{\on}{\operatorname}

\nc{\BA}{{\mathbb{A}}}
\nc{\BC}{{\mathbb{C}}}
\nc{\BD}{{\mathbb{D}}}
\nc{\BF}{{\mathbb{F}}}
\nc{\BG}{{\mathbb{G}}}
\nc{\BM}{{\mathbb{M}}}
\nc{\BN}{{\mathbb{N}}}
\nc{\BO}{{\mathbb{O}}}
\nc{\BQ}{{\mathbb{Q}}}
\nc{\BP}{{\mathbb{P}}}
\nc{\BR}{{\mathbb{R}}}
\nc{\BZ}{{\mathbb{Z}}}
\nc{\BS}{{\mathbb{S}}}
\nc{\BK}{{\mathbb{K}}}

\nc{\CA}{{\mathcal{A}}} \nc{\CB}{{\mathcal{B}}} \nc{\CalC}{{\mathcal
C}} \nc{\CalD}{{\mathcal D}} \nc{\CE}{{\mathcal{E}}}
\nc{\CF}{{\mathcal{F}}} \nc{\CG}{{\mathcal{G}}}
\nc{\CH}{{\mathcal{H}}} \nc{\CI}{{\mathcal{I}}}
\nc{\CK}{{\mathcal{K}}} \nc{\CL}{{\mathcal{L}}}
\nc{\CM}{{\mathcal{M}}} \nc{\CN}{{\mathcal{N}}}
\nc{\CO}{{\mathcal{O}}} \nc{\CP}{{\mathcal{P}}}
\nc{\CQ}{{\mathcal{Q}}} \nc{\CR}{{\mathcal{R}}}
\nc{\CS}{{\mathcal{S}}} \nc{\CT}{{\mathcal{T}}}
\nc{\CU}{{\mathcal{U}}} \nc{\CV}{{\mathcal{V}}}
\nc{\CW}{{\mathcal{W}}} \nc{\CX}{{\mathcal{X}}}
\nc{\CY}{{\mathcal{Y}}} \nc{\CZ}{{\mathcal{Z}}}

\nc{\fa}{{\mathfrak{a}}}
\nc{\fb}{{\mathfrak{b}}}
\nc{\fg}{{\mathfrak{g}}}
\nc{\fgl}{{\mathfrak{gl}}}
\nc{\fh}{{\mathfrak{h}}}
\nc{\fj}{{\mathfrak{j}}}
\nc{\fl}{{\mathfrak{l}}}
\nc{\fm}{{\mathfrak{m}}}
\nc{\fn}{{\mathfrak{n}}}
\nc{\fu}{{\mathfrak{u}}}
\nc{\fp}{{\mathfrak{p}}}
\nc{\frr}{{\mathfrak{r}}}
\nc{\fs}{{\mathfrak{s}}}
\nc{\ft}{{\mathfrak{t}}}
\nc{\fw}{{\mathfrak{w}}}
\nc{\fz}{{\mathfrak{z}}}

\nc{\fA}{{\mathfrak{A}}}
\nc{\fB}{{\mathfrak{B}}}
\nc{\fD}{{\mathfrak{D}}}
\nc{\fE}{{\mathfrak{E}}}
\nc{\fF}{{\mathfrak{F}}}
\nc{\fG}{{\mathfrak{G}}}
\nc{\fI}{{\mathfrak{I}}}
\nc{\fJ}{{\mathfrak{J}}}
\nc{\fK}{{\mathfrak{K}}}
\nc{\fL}{{\mathfrak{L}}}
\nc{\fM}{{\mathfrak{M}}}
\nc{\fN}{{\mathfrak{N}}}
\nc{\frP}{{\mathfrak{P}}}
\nc{\fQ}{{\mathfrak Q}}
\nc{\fR}{{\mathfrak R}}
\nc{\fS}{{\mathfrak S}}
\nc{\fT}{{\mathfrak{T}}}
\nc{\fU}{{\mathfrak{U}}}
\nc{\fW}{{\mathfrak{W}}}
\nc{\fY}{{\mathfrak{Y}}}
\nc{\fZ}{{\mathfrak{Z}}}

\nc{\ba}{{\mathbf{a}}}
\nc{\bb}{{\mathbf{b}}}
\nc{\bc}{{\mathbf{c}}}
\nc{\bd}{{\mathbf{d}}}
\nc{\be}{{\mathbf{e}}}
\nc{\bi}{{\mathbf{i}}}
\nc{\bj}{{\mathbf{j}}}
\nc{\bn}{{\mathbf{n}}}
\nc{\bp}{{\mathbf{p}}}
\nc{\bq}{{\mathbf{q}}}
\nc{\bu}{{\mathbf{u}}}
\nc{\bv}{{\mathbf{v}}}
\nc{\bw}{{\mathbf{w}}}
\nc{\bx}{{\mathbf{x}}}
\nc{\by}{{\mathbf{y}}}
\nc{\bz}{{\mathbf{z}}}

\nc{\bA}{{\mathbf{A}}}
\nc{\bB}{{\mathbf{B}}}
\nc{\bC}{{\mathbf{C}}}
\nc{\bD}{{\mathbf{D}}}
\nc{\bE}{{\mathbf{E}}}
\nc{\bI}{{\mathbf{I}}}
\nc{\bK}{{\mathbf{K}}}
\nc{\bH}{{\mathbf{H}}}
\nc{\bM}{{\mathbf{M}}}
\nc{\bN}{{\mathbf{N}}}
\nc{\bO}{{\mathbf{O}}}
\nc{\bQ}{{\mathbf Q}}
\nc{\bS}{{\mathbf{S}}}
\nc{\bT}{{\mathbf{T}}}
\nc{\bV}{{\mathbf{V}}}
\nc{\bW}{{\mathbf{W}}}
\nc{\bX}{{\mathbf{X}}}
\nc{\bP}{{\mathbf{P}}}
\nc{\bY}{{\mathbf{Y}}}
\nc{\bZ}{{\mathbf{Z}}}

\nc{\sA}{{\mathsf{A}}}
\nc{\sB}{{\mathsf{B}}}
\nc{\sC}{{\mathsf{C}}}
\nc{\sD}{{\mathsf{D}}}
\nc{\sF}{{\mathsf{F}}}
\nc{\sK}{{\mathsf{K}}}
\nc{\sM}{{\mathsf{M}}}
\nc{\sO}{{\mathsf{O}}}
\nc{\sQ}{{\mathsf{Q}}}
\nc{\sP}{{\mathsf{P}}}
\nc{\sT}{{\mathsf{T}}}
\nc{\sV}{{\mathsf{V}}}
\nc{\sW}{{\mathsf{W}}}
\nc{\sX}{{\mathsf{X}}}
\nc{\sZ}{{\mathsf{Z}}}
\nc{\sU}{{\mathsf{U}}}
\nc{\sS}{{\mathsf{S}}}

\nc{\sfb}{{\mathsf{b}}}
\nc{\sfc}{{\mathsf{c}}}
\nc{\sd}{{\mathsf{d}}}
\nc{\sg}{{\mathsf{g}}}
\nc{\sk}{{\mathsf{k}}}
\nc{\sfl}{{\mathsf{l}}}
\nc{\sfp}{{\mathsf{p}}}
\nc{\sr}{{\mathsf{r}}}
\nc{\st}{{\mathsf{t}}}
\nc{\sfu}{{\mathsf{u}}}
\nc{\sw}{{\mathsf{w}}}
\nc{\sz}{{\mathsf{z}}}
\nc{\sx}{{\mathsf{x}}}
\nc{\se}{{\mathsf{e}}}
\nc{\sfv}{{\mathsf{v}}}
\nc{\sff}{{\mathsf{f}}}

\nc{\bLambda}{{\boldsymbol{\Lambda}}}
\nc{\vv}{{\boldsymbol{v}}}
\nc{\Fl}{{{\mathcal F}\ell}}
\nc{\Gr}{{\on{Gr}}}
\nc{\CHH}{{\CH\!\!\CH}}
\nc{\lambdavee}{{\lambda^{\!\scriptscriptstyle\vee}}}
\nc{\alphavee}{\alpha^{\!\scriptscriptstyle\vee}}
\nc{\rhovee}{{\rho^{\!\scriptscriptstyle\vee}}}
\newcommand\iso{\,\vphantom{j^{X^2}}\smash{\overset{\sim}{\vphantom{\rule{0pt}{0.20em}}\smash{\longrightarrow}}}\,}
\nc{\oQM}{\vphantom{j^{X^2}}\smash{\overset{\circ}{\vphantom{\vstretch{0.7}{A}}\smash{\QM}}}}
\nc{\oZ}{{}^\dagger\!\vphantom{j^{X^2}}\smash{\overset{\circ}{\vphantom{\vstretch{0.7}{A}}\smash{Z}}}}
\nc{\odZ}{{}^\dagger\!\vphantom{j^{X^2}}\smash{\overset{\circ}{\vphantom{\vstretch{0.7}{A}}\smash{\mathfrak Z}}}^{c',c}}
\nc{\bdZ}{{}^\dagger\!\vphantom{j^{X^2}}\smash{\overset{\bullet}{\vphantom{\vstretch{0.7}{A}}\smash{\mathfrak Z}}}^{c',c}}
\nc{\oS}{\vphantom{j^{X^2}}\smash{\overset{\circ}{\vphantom{\vstretch{0.7}{A}}\smash{S}}}}
\nc{\buM}{\vphantom{j^{X^2}}\smash{\overset{\bullet}{\vphantom{\vstretch{0.7}{A}}\smash{M}}}}
\nc{\dW}{{}^\dagger\ol\CW{}}
\nc{\hW}{{}^\dagger\hat\CW{}}
\nc{\wW}{{}^\dagger\widetilde\CW{}}
\nc{\dZ}{{}^\dagger\!\fZ^{c',c}}
\nc{\dZc}{{}^\dagger\!\fZ^{c,c}}
\nc{\tZ}{{}^\dagger\!\tilde{Z}{}}
\nc{\hZ}{{}^\dagger\!\hat{Z}{}}

\nc{\ssl}{\mathfrak{sl}} \nc{\gl}{\mathfrak{gl}}
\nc{\wt}{\widetilde} \nc{\Sym}{\mathrm{Sym}} \nc{\Res}{\mathrm{Res}}
\nc{\sE}{{\mathsf{E}}} \nc{\bs}{{\mathbf{s}}}
\nc{\trig}{\mathrm{trig}} \nc{\rat}{\mathrm{rat}}
\nc{\sign}{\mathrm{sign}} \nc{\sL}{{\mathsf{L}}}
\nc{\fv}{{\mathfrak{v}}} \nc{\ad}{\mathrm{ad}}
\nc{\spsi}{{\mathsf{\psi}}} \nc{\sh}{{\mathsf{h}}}
\nc{\rtt}{\mathrm{rtt}} \nc{\qdet}{\mathrm{qdet}} \nc{\pt}{{\operatorname{pt}}}
\nc{\M}{\mathrm{M}} \nc{\Ker}{\mathrm{Ker}} \nc{\ssc}{\mathrm{sc}}
\nc{\loc}{\mathrm{loc}} \nc{\fra}{\mathrm{frac}}
\nc{\ddj}{\mathrm{DJ}} \nc{\End}{\mathrm{End}} \nc{\ev}{\mathrm{ev}}
\nc{\GL}{\mathrm{GL}}
\nc{\Or}{\mathrm{Or}}

\begin{document}
\title[Duality of Lusztig and RTT integral forms]
{Duality of Lusztig and RTT integral forms of $U_\vv(L\ssl_n)$}

\author[Alexander Tsymbaliuk]{Alexander Tsymbaliuk}
\address{A.T.: Yale University, Department of Mathematics, New Haven, CT 06511, USA}
\email{sashikts@gmail.com}

\begin{abstract}
We show that the Lusztig integral form is dual to the RTT integral form of the type $A$
quantum loop algebra with respect to the new Drinfeld pairing, by utilizing the shuffle
algebra realization of the former and the PBWD bases of the latter obtained in~\cite{t}.
\end{abstract}
\maketitle


\section{Introduction}


\subsection{Summary}
\

For a simple finite-dimensional Lie algebra $\fg$, the quantum function algebra is
dual to the Lusztig form $\sU_\vv(\fg)$ of the quantum group of $\fg$. For $\fg=\ssl_n$,
this is reflected by the duality between the Lusztig and the RTT integral forms of
$U_\vv(\ssl_n)$ with respect to the Drinfeld-Jimbo pairing. In this short note,
we establish an affine version of the above result for $\ssl_n$ replaced with
$\widehat{\ssl}_n$ and the Drinfeld-Jimbo pairing replaced with the new Drinfeld pairing.


\subsection{Outline of the paper}
\

$\bullet$
In Section~\ref{sec Classical quantuma affine}, we recall the quantum loop
(quantum affine with the trivial central charge) algebra $U_\vv(L\ssl_n)$
as well as its two integral forms: $\fU_\vv(L\ssl_n)$ (naturally arising in
the RTT presentation of~\cite{frt}) and $\sU_\vv(L\ssl_n)$
(Lusztig form defined in the Drinfeld-Jimbo presentation).
Both integral forms posses triangular decompositions, see
Propositions~\ref{Triangular for RTT integral form},~\ref{Triangular for Lusztig integral form},
generalizing the one for $U_\vv(L\ssl_n)$ of Proposition~\ref{Triangular decomposition}.
We also recall our constructions of the PBWD (Poincar\'{e}-Birkhoff-Witt-Drinfeld)
bases for the ``positive'' and ``negative'' subalgebras of both integral forms
established in~\cite{t}, see
Theorems~\ref{PBWD for RTT integral form},~\ref{PBWD for Lusztig integral form}.
Finally, in Section~\ref{ssec new Drinfeld pairing}, we recall
the new Drinfeld topological Hopf algebra structure and
the new Drinfeld pairing on $U_\vv(L\ssl_n)$.

$\bullet$
In Section~\ref{sec shuffle algebra realizations}, we recall the shuffle algebra
$S^{(n)}$, its two integral forms, and the shuffle algebra realizations of the
``positive'' subalgebras $U^>_\vv(L\ssl_n)$, see Theorem~\ref{shuffle isomorphism}
(first established in~\cite{n}), and of $\sU^>_\vv(L\ssl_n), \fU^>_\vv(L\ssl_n)$,
see Theorem~\ref{shuffle Grojnowski isomorphism} and Remark~\ref{shuffle RTT isomorphism},
established in~\cite{t}. Finally, we enlarge $S^{(n)}$ to the extended shuffle algebra
$S^{(n),\geq}$ by adjoining Cartan generators satisfying~(\ref{cartan shuffle action}),
thus obtaining the shuffle algebra realization~(\ref{extended shuffle isomorphism}) of
$U^\geq_\vv(L\ssl_n)$, and recall the formulas~(\ref{coproduct Cartan},~\ref{coproduct shuffle})
for the new Drinfeld coproduct on it, cf.~\cite[Proposition 3.5]{n}.

$\bullet$
In Section~\ref{sec main result}, we prove that the integral form $\sU_\vv(L\ssl_n)$
is dual to $\fU_\vv(L\ssl_n)$ with respect to the new Drinfeld pairing, see
Theorem~\ref{Main Theorem} and Remark~\ref{Opposite duality}, which constitutes
the main result of this note. Our proof is crucially based on the shuffle realizations of
Section~\ref{sec shuffle algebra realizations} as well as utilizes the entire family
of the PBWD bases of $\fU_\vv(L\ssl_n)$ of Theorem~\ref{PBWD for RTT integral form},
see Remark~\ref{importance of all decompositions}.


\subsection{Acknowledgments}
\

I am indebted to Pavel Etingof, Boris Feigin, Michael Finkelberg, Andrei Negu\c{t}
for stimulating discussions; to anonymous referee for useful suggestions.
This note was prepared during the author's visit to IH\'{E}S (Bures-sur-Yvette, France)
in the summer $2019$, sponsored by the European Research Council under
the European Union's Horizon $2020$ research and innovation program
(QUASIFT grant agreement $677368$); I am grateful to V.~Pestun for his invitation.

I gratefully acknowledge NSF Grants DMS-$1821185$, DMS-$2001247$, and DMS-$2037602$.


\section{Quantum loop algebra $U_\vv(L\ssl_n)$ and its integral forms}
\label{sec Classical quantuma affine}


\subsection{Quantum loop algebra $U_\vv(L\ssl_n)$}\label{ssec Drinfeld affine sl_n}
\

Let $I=\{1,\ldots,n-1\}$, $(c_{ij})_{i,j\in I}$ be the Cartan matrix of $\ssl_n$,
and $\vv$ be a formal variable. Following~\cite{d1}, define the quantum loop algebra
of $\ssl_n$ (in the new Drinfeld presentation), denoted by $U_\vv(L\ssl_n)$,
to be the associative $\BC(\vv)$-algebra generated by
  $\{e_{i,r},f_{i,r},\psi^\pm_{i,\pm s}\}_{i\in I}^{r\in \BZ, s\in \BN}$
with the following defining relations:
\begin{equation}\label{Aff 1}
  [\psi_i^\epsilon(z),\psi_j^{\epsilon'}(w)]=0,\
  \psi^\pm_{i,0}\cdot \psi^\mp_{i,0}=1,
\end{equation}
\begin{equation}\label{Aff 2}
  (z-\vv^{c_{ij}}w)e_i(z)e_j(w)=(\vv^{c_{ij}}z-w)e_j(w)e_i(z),
\end{equation}
\begin{equation}\label{Aff 3}
  (\vv^{c_{ij}}z-w)f_i(z)f_j(w)=(z-\vv^{c_{ij}}w)f_j(w)f_i(z),
\end{equation}
\begin{equation}\label{Aff 4}
  (z-\vv^{c_{ij}}w)\psi_i^\epsilon(z)e_j(w)=(\vv^{c_{ij}}z-w)e_j(w)\psi_i^\epsilon(z),
\end{equation}
\begin{equation}\label{Aff 5}
  (\vv^{c_{ij}}z-w)\psi^\epsilon_i(z)f_j(w)=(z-\vv^{c_{ij}}w)f_j(w)\psi^\epsilon_i(z),
\end{equation}
\begin{equation}\label{Aff 6}
  [e_i(z),f_j(w)]=
  \frac{\delta_{ij}}{\vv-\vv^{-1}}\delta\left(\frac{z}{w}\right)\left(\psi^+_i(z)-\psi^-_i(z)\right),
\end{equation}
\begin{equation}\label{Aff 7}
\begin{split}
  & e_i(z)e_j(w)=e_j(w)e_i(z)\ \mathrm{if}\ c_{ij}=0,\\
  & [e_i(z_1),[e_i(z_2),e_j(w)]_{\vv^{-1}}]_{\vv}+
    [e_i(z_2),[e_i(z_1),e_j(w)]_{\vv^{-1}}]_{\vv}=0 \ \mathrm{if}\ c_{ij}=-1,
\end{split}
\end{equation}
\begin{equation}\label{Aff 8}
\begin{split}
  & f_i(z)f_j(w)=f_j(w)f_i(z)\ \mathrm{if}\ c_{ij}=0,\\
  & [f_i(z_1),[f_i(z_2),f_j(w)]_{\vv^{-1}}]_{\vv}+
    [f_i(z_2),[f_i(z_1),f_j(w)]_{\vv^{-1}}]_{\vv}=0 \ \mathrm{if}\ c_{ij}=-1,
\end{split}
\end{equation}
where $[a,b]_x:=ab-x\cdot ba$ and the generating series are defined as follows:
\begin{equation*}
  e_i(z):=\sum_{r\in \BZ}{e_{i,r}z^{-r}},\
  f_i(z):=\sum_{r\in \BZ}{f_{i,r}z^{-r}},\
  \psi_i^{\pm}(z):=\sum_{s\geq 0}{\psi^\pm_{i,\pm s}z^{\mp s}},\
  \delta(z):=\sum_{r\in \BZ}{z^r}.
\end{equation*}

Let
  $U^{<}_\vv(L\ssl_n), U^{>}_\vv(L\ssl_n), U^{0}_\vv(L\ssl_n)$
be the $\BC(\vv)$-subalgebras of $U_\vv(L\ssl_n)$ generated
respectively by
  $\{f_{i,r}\}_{i\in I}^{r\in \BZ},
   \{e_{i,r}\}_{i\in I}^{r\in \BZ},
   \{\psi^\pm_{i,\pm s}\}_{i\in I}^{s\in \BN}$.
The following is standard (see e.g.~\cite[Theorem 2]{h}):

\begin{Prop}\label{Triangular decomposition}
(a) (Triangular decomposition of $U_\vv(L\ssl_n)$)
The multiplication map
\begin{equation*}
  m\colon
  U^{<}_\vv(L\ssl_n)\otimes_{\BC(\vv)} U^{0}_\vv(L\ssl_n)\otimes_{\BC(\vv)} U^{>}_\vv(L\ssl_n)
  \longrightarrow U_\vv(L\ssl_n)
\end{equation*}
is an isomorphism of $\BC(\vv)$-vector spaces.

\noindent
(b) The algebra $U^{>}_\vv(L\ssl_n)$ (resp.\ $U^{<}_\vv(L\ssl_n)$ and
$U^{0}_\vv(L\ssl_n)$) is isomorphic to the associative $\BC(\vv)$-algebra
generated by $\{e_{i,r}\}_{i\in I}^{r\in \BZ}$
(resp.\ $\{f_{i,r}\}_{i\in I}^{r\in \BZ}$ and $\{\psi^\pm_{i,\pm s}\}_{i\in I}^{s\in \BN}$)
with the defining relations~(\ref{Aff 2}, \ref{Aff 7})
(resp.~(\ref{Aff 3}, \ref{Aff 8}) and~(\ref{Aff 1})).
\end{Prop}


\subsection{RTT integral form $\fU_\vv(L\ssl_n)$ and its PBWD bases}\label{ssec RTT integral form}
\

Let $\{\alpha_i\}_{i=1}^{n-1}$ be the standard simple positive roots of $\ssl_n$,
and $\Delta^+$ be the set of positive roots:
  $\Delta^+=\{\alpha_j+\alpha_{j+1}+\ldots+\alpha_i\}_{1\leq j\leq i<n}$.
Consider the following total ordering ``$\leq$'' on $\Delta^+$:
\begin{equation}\label{order 1}
  \alpha_j+\alpha_{j+1}+\ldots+\alpha_i\leq \alpha_{j'}+\alpha_{j'+1}+\ldots+\alpha_{i'}
  \ \mathrm{iff}\ j<j'\ \mathrm{or}\ j=j',i\leq i'.
\end{equation}
This gives rise to the total ordering ``$\leq$'' on $\Delta^+\times \BZ$:
\begin{equation}\label{order 2}
  (\beta,r)\leq (\beta',r')\ \mathrm{iff}\
  \beta<\beta'\ \mathrm{or}\ \beta=\beta', r\leq r'.
\end{equation}

For any $1\leq j\leq i\leq n-1$ and $r\in \BZ$, we choose a \emph{decomposition}
\begin{equation}\label{decomposition}
  \unl{r}=\unl{r}(\alpha_j+\ldots+\alpha_i,r)=(r_j,\ldots,r_i)\in \BZ^{i-j+1}
  \ \mathrm{such\ that}\ r_j+\ldots+r_i=r.
\end{equation}
A particular example of such a decomposition is
\begin{equation}\label{simplest decomposition}
  \unl{r}^{(0)}=\unl{r}^{(0)}(\alpha_j+\ldots+\alpha_i,r)=(r,0,\ldots,0).
\end{equation}

Following~\cite[(2.11, 2.18)]{t}, define the elements
$\wt{e}_{\beta,\unl{r}}\in U^{>}_\vv(L\ssl_n)$ and
$\wt{f}_{\beta,\unl{r}}\in U^{<}_\vv(L\ssl_n)$ via
\begin{equation}\label{integral higher roots}
\begin{split}
  & \wt{e}_{\alpha_j+\alpha_{j+1}+\ldots+\alpha_i,\unl{r}}:=
    (\vv-\vv^{-1})[\cdots[[e_{j,r_j},e_{j+1,r_{j+1}}]_\vv,e_{j+2,r_{j+2}}]_\vv,\cdots,e_{i,r_i}]_\vv,\\
  & \wt{f}_{\alpha_j+\alpha_{j+1}+\ldots+\alpha_i,\unl{r}}:=
    (\vv-\vv^{-1})[\cdots[[f_{j,r_j},f_{j+1,r_{j+1}}]_\vv,f_{j+2,r_{j+2}}]_\vv,\cdots,f_{i,r_i}]_\vv.
\end{split}
\end{equation}
In the special case $\unl{r}(\beta,r)=\unl{r}^{(0)}(\beta,r)$,
see~(\ref{simplest decomposition}), we shall denote
$\wt{e}_{\beta,\unl{r}}, \wt{f}_{\beta,\unl{r}}$ simply by
$\wt{e}_{\beta,r}, \wt{f}_{\beta,r}$.

The \emph{RTT integral form} $\fU_\vv(L\ssl_n)$ is the
$\BC[\vv,\vv^{-1}]$-subalgebra of $U_\vv(L\ssl_n)$ generated by
  $\{\wt{e}_{\beta,r},\wt{f}_{\beta,r},\psi^\pm_{i,\pm s}\}_{i\in I,\beta\in \Delta^+}^{r\in \BZ, s\in \BN}$.
Let $\fU^{<}_\vv(L\ssl_n),\fU^{>}_\vv(L\ssl_n)$, and $\fU^{0}_\vv(L\ssl_n)$
be the $\BC[\vv,\vv^{-1}]$-subalgebras of $\fU_\vv(L\ssl_n)$ generated by
  $\{\wt{f}_{\beta,r}\}_{\beta\in \Delta^+}^{r\in \BZ},
   \{\wt{e}_{\beta,r}\}_{\beta\in \Delta^+}^{r\in \BZ}$,
and
  $\{\psi^\pm_{i,\pm s}\}_{i\in I}^{s\in \BN}$, respectively.

\begin{Rem}\label{integral form via RTT}
The name ``\emph{RTT integral form}'' is motivated by the following two observations:

\noindent
(a) Due to Theorem~\ref{PBWD for RTT integral form} below, we have
  $\fU_\vv(L\ssl_n)\otimes_{\BC[\vv,\vv^{-1}]} \BC(\vv)\simeq U_\vv(L\ssl_n)$.

\noindent
(b) Due to~\cite[Proposition 3.20]{ft}, $\fU_\vv(L\ssl_n)$ coincides with
the $\Upsilon$-preimage of $\fU^\rtt_\vv(L\gl_n)$, where $\fU^\rtt_\vv(L\gl_n)$
is the RTT integral form of the quantum loop algebra of $\gl_n$~\cite{frt}
(cf.~\cite[\S3(ii)]{ft}), while
  $\Upsilon\colon U_\vv(L\ssl_n)\hookrightarrow
   \fU^\rtt_\vv(L\gl_n)\otimes_{\BC[\vv,\vv^{-1}]} \BC(\vv)$
is the $\BC(\vv)$-algebra embedding of~\cite{df}.
\end{Rem}

As before, fix a decomposition $\unl{r}(\beta,r)$ for each pair
$(\beta,r)\in \Delta^+\times \BZ$. We order
  $\{\wt{e}_{\beta,\unl{r}(\beta,r)}\}_{\beta\in \Delta^+}^{r\in \BZ}$
with respect to~(\ref{order 2}), while
  $\{\wt{f}_{\beta,\unl{r}(\beta,r)}\}_{\beta\in \Delta^+}^{r\in \BZ}$
are ordered with respect to the opposite ordering on $\Delta^+\times \BZ$.
Finally, choose any total ordering of $\{\psi_{i,r}\}_{i\in I}^{r\in \BZ}$
defined via
  $\psi_{i,r}:=
   \begin{cases}
     \psi^+_{i,r}, & \text{if } r\geq 0 \\
     \psi^-_{i,r}, & \text{if } r<0
   \end{cases}$.
Having specified these three total orderings, elements $F\cdot H\cdot E$
with $F,E,H$ being ordered monomials in
  $\{\wt{f}_{\beta,\unl{r}(\beta,r)}\}_{\beta\in \Delta^+}^{r\in \BZ}$,
  $\{\wt{e}_{\beta,\unl{r}(\beta,r)}\}_{\beta\in \Delta^+}^{r\in \BZ}$,
  $\{\psi_{i,r}\}_{i\in I}^{r\in \BZ}$
(note that we allow negative powers of $\psi_{i,0}$), respectively,
are called the \emph{ordered PBWD monomials} (in the corresponding generators).


The following was established in~\cite[Theorems 2.15, 2.17, 2.19, 2.22]{t},
cf.~\cite[Theorem~3.24]{ft}:

\begin{Thm}\label{PBWD for RTT integral form}
Fix a decomposition $\unl{r}(\beta,r)$ for every pair $(\beta,r)\in \Delta^+\times \BZ$.

\noindent
(a1) The ordered PBWD monomials in
  $\{\wt{f}_{\beta,\unl{r}(\beta,r)}, \psi_{i,r}, \wt{e}_{\beta,\unl{r}(\beta,r)}\}_{i\in I,\beta\in \Delta^+}^{r\in \BZ}$
form a basis of the free $\BC[\vv,\vv^{-1}]$-module $\fU_\vv(L\ssl_n)$.


\noindent
(a2) The ordered PBWD monomials in
  $\{\wt{f}_{\beta,\unl{r}(\beta,r)}, \psi_{i,r}, \wt{e}_{\beta,\unl{r}(\beta,r)}\}_{i\in I,\beta\in \Delta^+}^{r\in \BZ}$
form a $\BC(\vv)$-basis of $U_\vv(L\ssl_n)$.


\noindent
(b1) The ordered PBWD monomials in
  $\{\wt{e}_{\beta,\unl{r}(\beta,r)}\}_{\beta\in \Delta^+}^{r\in \BZ}$
form a basis of the free $\BC[\vv,\vv^{-1}]$-module $\fU^>_\vv(L\ssl_n)$.


\noindent
(b2) The ordered PBWD monomials in
  $\{\wt{e}_{\beta,\unl{r}(\beta,r)}\}_{\beta\in \Delta^+}^{r\in \BZ}$
form a $\BC(\vv)$-basis of $U^>_\vv(L\ssl_n)$.


\noindent
(c1) The ordered PBWD monomials in
  $\{\wt{f}_{\beta,\unl{r}(\beta,r)}\}_{\beta\in \Delta^+}^{r\in \BZ}$
form a basis of the free $\BC[\vv,\vv^{-1}]$-module $\fU^<_\vv(L\ssl_n)$.


\noindent
(c2) The ordered PBWD monomials in
  $\{\wt{f}_{\beta,\unl{r}(\beta,r)}\}_{\beta\in \Delta^+}^{r\in \BZ}$
form a $\BC(\vv)$-basis of $U^<_\vv(L\ssl_n)$.


\noindent
(d1) The ordered PBWD monomials in $\{\psi_{i,r}\}_{i\in I}^{r\in \BZ}$
form a basis of the free $\BC[\vv,\vv^{-1}]$-module $\fU^0_\vv(L\ssl_n)$.


\noindent
(d2) The ordered PBWD monomials in $\{\psi_{i,r}\}_{i\in I}^{r\in \BZ}$
form a $\BC(\vv)$-basis of $U^0_\vv(L\ssl_n)$.
\end{Thm}

This result together with Proposition~\ref{Triangular decomposition}
implies the triangular decomposition of $\fU_\vv(L\ssl_n)$:

\begin{Prop}\label{Triangular for RTT integral form}
The multiplication map
\begin{equation*}
  m\colon
  \fU^{<}_\vv(L\ssl_n)\otimes_{\BC[\vv,\vv^{-1}]}
  \fU^{0}_\vv(L\ssl_n)\otimes_{\BC[\vv,\vv^{-1}]}
  \fU^{>}_\vv(L\ssl_n)\longrightarrow \fU_\vv(L\ssl_n)
\end{equation*}
is an isomorphism of (free) $\BC[\vv,\vv^{-1}]$-modules.
\end{Prop}


\subsection{Lusztig integral form $\sU_\vv(L\ssl_n)$ and its PBWD basis}\label{ssec Lusztig integral form}
\

To introduce the Lusztig integral form, we recall the Drinfeld-Jimbo realization of
$U_\vv(L\ssl_n)$. Let $\wt{I}=I\cup\{i_0\}$ be the vertex set of the extended
Dynkin diagram and $(c_{ij})_{i,j\in \wt{I}}$ be the extended Cartan matrix.
The \emph{Drinfeld-Jimbo quantum loop algebra of $\ssl_n$}, denoted by
$U^{\ddj}_\vv(L\ssl_n)$, is the associative $\BC(\vv)$-algebra generated by
  $\{E_i, F_i, K^{\pm 1}_i\}_{i\in \widetilde{I}}$
with the following defining relations:
\begin{equation}\label{DJ 1}
  [K_i,K_j]=0,\ K_i^{\pm 1}\cdot K_i^{\mp 1}=1,\
  \prod_{i\in \wt{I}} K_i=1,
\end{equation}
\begin{equation}\label{DJ 2}
  K_iE_j=\vv^{c_{ij}}E_jK_i,\ K_iF_j=\vv^{-c_{ij}}F_jK_i,\
  [E_i,F_j]=\delta_{ij}\frac{K_i-K_i^{-1}}{\vv-\vv^{-1}},
\end{equation}
\begin{equation}\label{DJ 4}
  E_iE_j=E_jE_i,\ F_iF_j=F_jF_i\ \mathrm{if}\ c_{ij}=0,
\end{equation}
\begin{equation}\label{DJ 5}
  [E_i,[E_i,E_j]_{\vv^{-1}}]_{\vv}=0,\
  [F_i,[F_i,F_j]_{\vv^{-1}}]_{\vv}=0 \ \mathrm{if}\ c_{ij}=-1.
\end{equation}

The following result is due to~\cite{d1}:

\begin{Prop}\label{identification of DJ and Dr}
There is a $\BC(\vv)$-algebra isomorphism
$U^\ddj_\vv(L\ssl_n)\iso U_\vv(L\ssl_n)$, such that
\begin{equation*}
  E_i\mapsto e_{i,0},\ F_i\mapsto f_{i,0},\
  K^{\pm 1}_i\mapsto \psi^\pm_{i,0}\ \mathrm{for}\ i\in I,
\end{equation*}
\begin{equation*}
  E_{i_0}\mapsto
  (-\vv)^{-n}\cdot (\psi^+_{1,0}\cdots\psi^+_{n-1,0})^{-1}\cdot
  [\cdots[f_{1,1},f_{2,0}]_\vv,\cdots,f_{n-1,0}]_\vv,
\end{equation*}
\begin{equation*}
  F_{i_0}\mapsto
  (-\vv)^n\cdot [e_{n-1,0},\cdots,[e_{2,0},e_{1,-1}]_{\vv^{-1}}\cdots]_{\vv^{-1}}\cdot
  \psi^+_{1,0}\cdots\psi^+_{n-1,0}.
\end{equation*}
\end{Prop}

For $k\in \BN$, set
  $[k]_\vv:=\frac{\vv^k-\vv^{-k}}{\vv-\vv^{-1}},
   [k]_\vv!:=\prod_{\ell=1}^k [\ell]_\vv$.
For $i\in \wt{I}, k\in \BN$, define the divided powers
\begin{equation}\label{Lusztig divided}
  E_i^{(k)}:=\frac{E_i^k}{[k]_\vv!}\ \
  \mathrm{and}\ \ F_i^{(k)}:=\frac{F_i^k}{[k]_\vv!}.
\end{equation}

The \emph{Lusztig integral form} $\sU^\ddj_\vv(L\ssl_n)$ is the
$\BC[\vv,\vv^{-1}]$-subalgebra of $U^\ddj_\vv(L\ssl_n)$ generated by
$\{E_i^{(k)},F_i^{(k)},K_i^{\pm 1}\}_{i\in \wt{I}}^{k\in \BN}$.
In view of Proposition~\ref{identification of DJ and Dr}, it gives rise to
the $\BC[\vv,\vv^{-1}]$-subalgebra $\sU_\vv(L\ssl_n)$ of $U_\vv(L\ssl_n)$
which shall be referred to as the \emph{Lusztig integral form} of $U_\vv(L\ssl_n)$.

Let us now recall a more explicit description of $\sU_\vv(L\ssl_n)$.
For $i\in I, r\in \BZ, k\in \BZ_{>0}$, define
\begin{equation}\label{Cartan divided}
  {\psi^+_{i,0}; r \brack k}:=
  \prod_{\ell=1}^{k}\frac{\psi^+_{i,0}\vv^{r-\ell+1}-\psi^-_{i,0}\vv^{-r+\ell-1}}{\vv^{\ell}-\vv^{-\ell}}.
\end{equation}
We also define the pairwise commuting generators $\{h_{i,r}\}_{i\in I}^{r\ne 0}$ via
\begin{equation}\label{h-generators}
   \psi^\pm_i(z)=
   \psi^\pm_{i,0}\cdot \exp\left(\pm(\vv-\vv^{-1})\sum_{r>0}h_{i,\pm r}z^{\mp r}\right).
\end{equation}
Finally, for $i\in I, r\in \BZ, k\in \BN$, we define the divided powers
\begin{equation}\label{devided power}
  \se_{i,r}^{(k)}:=\frac{e_{i,r}^k}{[k]_\vv!}\ \ \mathrm{and}\ \
  \sff_{i,r}^{(k)}:=\frac{f_{i,r}^k}{[k]_\vv!}.
\end{equation}
Let $\sU^<_\vv(L\ssl_n), \sU^>_\vv(L\ssl_n)$, and $\sU^0_\vv(L\ssl_n)$ be
the $\BC[\vv,\vv^{-1}]$-subalgebras of $U_\vv(L\ssl_n)$ generated by
  $\{\sff_{i,r}^{(k)}\}_{i\in I}^{r\in \BZ,k\in \BN},
   \{\se_{i,r}^{(k)}\}_{i\in I}^{r\in \BZ,k\in \BN}$,
and
  $\{\psi^\pm_{i,0}, \frac{h_{i,\pm k}}{[k]_\vv}, {\psi^+_{i,0}; r \brack k}\}_{i\in I}^{r\in \BZ,k\in \BZ_{>0}}$,
respectively.

\begin{Rem}\label{relation to Grojnowski}
The subalgebra $\sU^>_\vv(L\ssl_n)\subset U^>_\vv(L\ssl_n)$
was first considered in~\cite{g}.
\end{Rem}

The following triangular decomposition of $\sU_\vv(L\ssl_n)$
is due to~\cite[Proposition 6.1]{cp}:

\begin{Prop}\label{Triangular for Lusztig integral form}
(a) $\sU^<_\vv(L\ssl_n),\sU^0_\vv(L\ssl_n),\sU^>_\vv(L\ssl_n)$ are
$\BC[\vv,\vv^{-1}]$-subalgebras of $\sU_\vv(L\ssl_n)$.

\noindent
(b) (Triangular decomposition of $\sU_\vv(L\ssl_n)$)
The multiplication map
\begin{equation*}
  m\colon
  \sU^{<}_\vv(L\ssl_n)\otimes_{\BC[\vv,\vv^{-1}]}
  \sU^{0}_\vv(L\ssl_n)\otimes_{\BC[\vv,\vv^{-1}]}
  \sU^{>}_\vv(L\ssl_n)\longrightarrow \sU_\vv(L\ssl_n)
\end{equation*}
is an isomorphism of (free) $\BC[\vv,\vv^{-1}]$-modules.
\end{Prop}

Following~(\ref{simplest decomposition},~\ref{integral higher roots}),
define the elements $e_{\beta,r}\in U^{>}_\vv(L\ssl_n)$ and
$f_{\beta,r}\in U^{<}_\vv(L\ssl_n)$ via
\begin{equation}\label{integral higher roots 2}
\begin{split}
  & e_{\alpha_j+\alpha_{j+1}+\ldots+\alpha_i,r}:=
    [\cdots[[e_{j,r},e_{j+1,0}]_\vv,e_{j+2,0}]_\vv,\cdots,e_{i,0}]_\vv,\\
  & f_{\alpha_j+\alpha_{j+1}+\ldots+\alpha_i,r}:=
    [\cdots[[f_{j,r},f_{j+1,0}]_\vv,f_{j+2,0}]_\vv,\cdots,f_{i,0}]_\vv.
\end{split}
\end{equation}
For $\beta\in \Delta^+,r\in \BZ,k\in \BN$, we define the divided powers
\begin{equation}\label{devided higher root power}
  \se_{\beta,r}^{(k)}:=\frac{e_{\beta,r}^k}{[k]_\vv!}\ \ \mathrm{and}\ \
  \sff_{\beta,r}^{(k)}:=\frac{f_{\beta,r}^k}{[k]_\vv!}.
\end{equation}
Note that $\se_{\beta,r}^{(k)}\in \sU^>_\vv(L\ssl_n)$ and
$\sff_{\beta,r}^{(k)}\in \sU^<_\vv(L\ssl_n)$ for any $\beta,r,k$ as above,
due to~\cite[Theorem~6.6]{l}.


Evoking~(\ref{order 2}), the monomials of the form
  $\prod\limits_{(\beta,r)\in \Delta^+\times \BZ}^{\rightarrow} \se_{\beta,r}^{(k_{\beta,r})}$
and
  $\prod\limits_{(\beta,r)\in \Delta^+\times \BZ}^{\leftarrow} \sff_{\beta,r}^{(k_{\beta,r})}$
(with $k_{\beta,r}\in \BN$ and only finitely many of them being nonzero)
are called the \emph{ordered PBWD monomials} of $\sU^{>}_\vv(L\ssl_n)$
and $\sU^{<}_\vv(L\ssl_n)$, respectively.
The following result was established in~\cite{t}:

\begin{Thm}\cite[Theorem 8.5]{t}\label{PBWD for Lusztig integral form}
The ordered PBWD monomials form bases of the free $\BC[\vv,\vv^{-1}]$-modules
$\sU^>_\vv(L\ssl_n)$ and $\sU^<_\vv(L\ssl_n)$, respectively.
%
\end{Thm}


\subsection{New Drinfeld Hopf algebra structure and Hopf pairing}\label{ssec new Drinfeld pairing}
\

Let us first recall the general notion of a Hopf pairing, following~\cite[\S3]{krt}.
Given two Hopf algebras $A$ and $B$ over a field $\mathsf{k}$, the bilinear map
\begin{equation*}
  \varphi\colon A \times B\longrightarrow \mathsf{k}
\end{equation*}
is called a \emph{Hopf pairing} if it satisfies the following properties
(for any $a,a'\in A$ and $b,b'\in B$):
\begin{equation}\label{Hopf pairing 1}
  \varphi(a,bb')=\varphi(a_{(1)},b)\varphi(a_{(2)},b'),\
  \varphi(aa',b)=\varphi(a,b_{(2)})\varphi(a',b_{(1)}),
\end{equation}
\begin{equation}\label{Hopf pairing 2}
  \varphi(a,1_B)=\epsilon_A(a),\
  \varphi(1_A,b)=\epsilon_B(b),\
  \varphi(S_A(a),S_B(b))=\varphi(a,b),
\end{equation}
where we use the Sweedler notation for the coproduct:
$\Delta(x)=x_{(1)}\otimes x_{(2)}$.

Following~\cite[Theorem 2.1]{di}, we endow $U_\vv(L\ssl_n)$ with
the new Drinfeld topological Hopf algebra structure by defining
the coproduct $\Delta$, the counit $\epsilon$, and the antipode $S$ as follows:
\begin{equation*}\label{coproduct}
  \Delta\colon
  \psi^\pm_i(z)\mapsto \psi^\pm_i(z)\otimes \psi^\pm_i(z),\
  e_i(z)\mapsto e_i(z)\otimes 1 + \psi^-_i(z)\otimes e_i(z),\
  f_i(z)\mapsto 1\otimes f_i(z) + f_i(z)\otimes \psi_i^+(z),
\end{equation*}
\begin{equation*}\label{counit}
  \epsilon\colon
  e_i(z)\mapsto 0,\
  f_i(z)\mapsto 0,\
  \psi^{\pm}_i(z)\mapsto 1,
\end{equation*}
\begin{equation*}\label{antipode}
  S\colon
  e_i(z)\mapsto -\psi^-_i(z)^{-1}e_i(z),\
  f_i(z)\mapsto -f_i(z)\psi^+_i(z)^{-1},\
  \psi^\pm_i(z)\mapsto \psi^\pm_i(z)^{-1}.
\end{equation*}
Thus, the $\BC(\vv)$-subalgebras $U^{\leq}_\vv(L\ssl_n),U^{\geq}_\vv(L\ssl_n)$
generated by
  $\{f_{i,r},\psi^+_{i,s},(\psi^+_{i,0})^{-1}\}_{i\in I}^{r\in \BZ,s\in \BN}$
and
  $\{e_{i,r},\psi^-_{i,-s},(\psi^-_{i,0})^{-1}\}_{i\in I}^{r\in \BZ,s\in \BN}$,
respectively, are actually Hopf subalgebras of $(U_\vv(L\ssl_n),\Delta,S,\epsilon)$.

The following is well-known (see e.g.~\cite[\S9.3]{g}, cf.~\cite[Propositions 2.27, 2.30]{n}):

\begin{Prop}\label{Drinfeld double sln}
The assignment
\begin{equation}\label{pairing formulas}
\begin{split}
  & \varphi(e_i(z), f_j(w))=\frac{\delta_{i,j}}{\vv-\vv^{-1}}\delta\left(\frac{z}{w}\right),\
    \varphi(\psi^-_i(z),\psi^+_j(w))=\frac{\vv^{c_{ij}}z-w}{z-\vv^{c_{ij}}w },\\
  & \varphi(e_i(z),\psi^+_j(w))=0,\ \varphi(\psi^-_j(w),f_i(z))=0,
\end{split}
\end{equation}
gives rise to a non-degenerate Hopf algebra pairing
  $\varphi\colon U^\geq_\vv(L\ssl_n)\times U^\leq_\vv(L\ssl_n)\to \BC(\vv)$.
\end{Prop}


\section{Shuffle algebra $S^{(n)}$ and its integral forms}
\label{sec shuffle algebra realizations}


\subsection{Shuffle algebra $S^{(n)}$}\label{ssec usual shuffle algebra}
\

Let $\Sigma_k$ denote the symmetric group in $k$ elements, and set
  $\Sigma_{(k_1,\ldots,k_{n-1})}:=\Sigma_{k_1}\times \cdots\times \Sigma_{k_{n-1}}$
for $k_1,\ldots,k_{n-1}\in \BN$.
Consider an $\BN^I$-graded $\BC(\vv)$-vector space
  $\BS^{(n)}=
   \underset{\underline{k}=(k_1,\ldots,k_{n-1})\in \BN^{I}}
   \bigoplus\BS^{(n)}_{\underline{k}}$,
where $\BS^{(n)}_{(k_1,\ldots,k_{n-1})}$ consists of $\Sigma_{\unl{k}}$-symmetric
rational functions in the variables $\{x_{i,r}\}_{i\in I}^{1\leq r\leq k_i}$.
Define $\zeta_{i,j}(z):=\frac{z-\vv^{-c_{ij}}}{z-1}$ for $i,j\in I$.
Let us introduce the bilinear \emph{shuffle product} $\star$ on $\BS^{(n)}$:
given  $F\in \BS^{(n)}_{\underline{k}}$ and $G\in \BS^{(n)}_{\underline{\ell}}$,
define $F\star G\in \BS^{(n)}_{\underline{k}+\underline{\ell}}$ via
\begin{equation}\label{shuffle product}
\begin{split}
  & (F\star G)(x_{1,1},\ldots,x_{1,k_1+\ell_1};\ldots;x_{n-1,1},\ldots, x_{n-1,k_{n-1}+\ell_{n-1}}):=
    \unl{k}!\cdot\unl{\ell}!\times\\
  & \Sym_{\Sigma_{\unl{k}+\unl{\ell}}}
    \left(F\left(\{x_{i,r}\}_{i\in I}^{1\leq r\leq k_i}\right) G\left(\{x_{i',r'}\}_{i'\in I}^{k_{i'}<r'\leq k_{i'}+\ell_{i'}}\right)\cdot
    \prod_{i\in I}^{i'\in I}\prod_{r\leq k_i}^{r'>k_{i'}}\zeta_{i,i'}(x_{i,r}/x_{i',r'})\right).
\end{split}
\end{equation}
Here, $\unl{k}!=\prod_{i\in I}k_i!$, while for
  $f\in \BC(\{x_{i,1},\ldots,x_{i,m_i}\}_{i\in I})$
we define its \emph{symmetrization} via
\begin{equation*}
  \Sym_{\Sigma_{\unl{m}}}(f)\left(\{x_{i,1},\ldots,x_{i,m_i}\}_{i\in I}\right):=
  \frac{1}{\unl{m}!}\cdot
  \sum_{(\sigma_1,\ldots,\sigma_{n-1})\in \Sigma_{\unl{m}}}
  f\left(\{x_{i,\sigma_i(1)},\ldots,x_{i,\sigma_i(m_i)}\}_{i\in I}\right).
\end{equation*}
This endows $\BS^{(n)}$ with a structure of an associative unital algebra
with the unit $\textbf{1}\in \BS^{(n)}_{(0,\ldots,0)}$.

We will be interested only in the subspace of $\BS^{(n)}$ defined by
the \emph{pole} and \emph{wheel conditions}:

\noindent
$\bullet$
We say that $F\in \BS^{(n)}_{\underline{k}}$ satisfies the \emph{pole conditions} if
\begin{equation}\label{pole conditions}
  F=
  \frac{f(x_{1,1},\ldots,x_{n-1,k_{n-1}})}
       {\prod_{i=1}^{n-2}\prod_{r\leq k_i}^{r'\leq k_{i+1}}(x_{i,r}-x_{i+1,r'})},\
  \mathrm{where}\ f\in (\BC(\vv)[\{x_{i,r}^{\pm 1}\}_{i\in I}^{1\leq r\leq k_i}])^{\Sigma_{\unl{k}}}.
\end{equation}

\noindent
$\bullet$
We say that $F\in \BS^{(n)}_{\underline{k}}$ satisfies the \emph{wheel conditions} if
\begin{equation}\label{wheel conditions}
  F(\{x_{i,r}\})=0\ \mathrm{once}\ x_{i,r_1}=\vv x_{i+\epsilon,s}=\vv^2 x_{i,r_2}\
  \mathrm{for\ some}\ \epsilon, i, r_1, r_2, s,
\end{equation}
where
  $\epsilon\in \{\pm 1\},\ i,i+\epsilon\in I,\
   1\leq r_1,r_2\leq k_i,\ 1\leq s\leq k_{i+\epsilon}$.

Let $S^{(n)}_{\underline{k}}\subset \BS^{(n)}_{\underline{k}}$ denote
the subspace of all elements $F$ satisfying these two conditions and set
  $S^{(n)}:=\underset{\underline{k}\in \BN^{I}}\bigoplus S^{(n)}_{\underline{k}}$.
It is straightforward to check that the subspace $S^{(n)}\subset\BS^{(n)}$ is $\star$-closed.

The resulting associative $\BC(v)$-algebra
$\left(S^{(n)},\star\right)$ shall be called the \emph{shuffle algebra}.


\subsection{Shuffle algebra realizations}\label{ssec shuffle realizations}
\

The \emph{shuffle algebra} $\left(S^{(n)},\star\right)$ is related
to $U^>_\vv(L\ssl_n)$ via the following result of~\cite{t} (cf.~\cite{n}):

\begin{Thm}\label{shuffle isomorphism}
The assignment $e_{i,r}\mapsto x_{i,1}^r\ (i\in I,r\in \BZ)$ gives rise to a
$\BC(\vv)$-algebra isomorphism $\Psi\colon U_\vv^{>}(L\ssl_n)\iso S^{(n)}$.
\end{Thm}

\begin{Rem}
This result was manifestly used in~\cite{t} to establish
parts (b2, c2) of Theorem~\ref{PBWD for RTT integral form}.
\end{Rem}

The proof of Theorem~\ref{shuffle isomorphism} in~\cite{t} crucially utilized the
\emph{specialization maps} $\phi_{\unl{d}}$~\cite[(3.12)]{t}, which we recall next.
For a positive root $\beta=\alpha_j+\alpha_{j+1}+\ldots+\alpha_i$, define
$j(\beta):=j,i(\beta):=i$, and let $[\beta]$ denote the integer interval $[j(\beta);i(\beta)]$.
Consider a collection of the intervals $\{[\beta]\}_{\beta\in \Delta^+}$
each taken with a multiplicity $d_{\beta}\in \BN$ and ordered with respect
to the total ordering~(\ref{order 1}) (the order inside each group is irrelevant).
Define $\unl{\ell}\in \BN^I$ via
  $\sum_{i\in I} \ell_i \alpha_i=\sum_{\beta\in \Delta^+} d_{\beta}\beta$.

Let us now define the \emph{specialization map} $\phi_{\unl{d}}$
(here, $\unl{d}$ denotes the collection $\{d_\beta\}_{\beta\in \Delta^+}$)
\begin{equation}\label{specialization map}
  \phi_{\unl{d}}\colon S^{(n)}_{\unl{\ell}}\longrightarrow
  \BC(\vv)[\{y_{\beta,s}^{\pm 1}\}_{\beta\in \Delta^+}^{1\leq s\leq d_\beta}].
\end{equation}
Split the variables $\{x_{i,r}\}_{i\in I}^{1\leq r\leq \ell_i}$ into
$\sum_{\beta\in \Delta^+} d_\beta$ groups corresponding to the above intervals,
and specialize those in the $s$-th copy of $[\beta]$ to
  $\vv^{-j(\beta)}\cdot y_{\beta,s},\ldots,\vv^{-i(\beta)}\cdot y_{\beta,s}$
in the natural order
(the variable $x_{k,r}$ gets specialized to $\vv^{-k}y_{\beta,s}$). For
  $F=\frac{f(x_{1,1},\ldots,x_{n-1,\ell_{n-1}})}
   {\prod_{i=1}^{n-2}\prod_{1\leq r\leq \ell_i}^{1\leq r'\leq \ell_{i+1}} (x_{i,r}-x_{i+1,r'})}
   \in S^{(n)}_{\unl{\ell}}$,
we define $\phi_{\unl{d}}(F)$ as the corresponding specialization of $f$.
Note that $\phi_{\unl{d}}(F)$ is independent of our splitting of the
variables $\{x_{i,r}\}_{i\in I}^{1\leq r\leq \ell_i}$ into groups and is
symmetric in $\{y_{\beta,s}\}_{s=1}^{d_\beta}$ for any $\beta$.

Following~\cite[Definition 8.6]{t}, an element $F\in S^{(n)}_{\unl{k}}$
is called \textbf{good} if the following holds:

\noindent
$\bullet$ $F$ is of the form~(\ref{pole conditions}) with
  $f\in \BC[\vv,\vv^{-1}][\{x_{i,r}^{\pm 1}\}_{i\in I}^{1\leq r\leq k_i}]^{\Sigma_{\unl{k}}}$;

\noindent
$\bullet$ $\phi_{\unl{d}}(F)$ is divisible by
  $(\vv-\vv^{-1})^{\sum_{\beta\in \Delta^+} d_\beta(i(\beta)-j(\beta))}$
for any $\unl{d}$ such that $\sum_{i\in I} k_i\alpha_i=\sum_{\beta\in \Delta^+} d_\beta \beta$.

Let $\sS^{(n)}_{\unl{k}}\subset S^{(n)}_{\unl{k}}$ denote the
$\BC[\vv,\vv^{-1}]$-submodule of all \emph{good} elements. Set
  $\sS^{(n)}:=\underset{\unl{k}\in \BN^I}\bigoplus \sS^{(n)}_{\underline{k}}$.

\begin{Thm}\cite[Theorem 8.8]{t}\label{shuffle Grojnowski isomorphism}
The $\BC(\vv)$-algebra isomorphism $\Psi\colon U_\vv^{>}(L\ssl_n)\iso S^{(n)}$
of Theorem~\ref{shuffle isomorphism} gives rise to a $\BC[\vv,\vv^{-1}]$-algebra
isomorphism $\Psi\colon \sU_\vv^{>}(L\ssl_n)\iso \sS^{(n)}$.
\end{Thm}


\begin{Rem}\label{shuffle RTT isomorphism}
In~\cite[Theorem 3.34]{t}, we also established the shuffle realization of
$\fU_\vv^{>}(L\ssl_n)$ by showing that the isomorphism $\Psi$ of
Theorem~\ref{shuffle isomorphism} gives rise to a $\BC[\vv,\vv^{-1}]$-algebra
isomorphism $\Psi\colon \fU_\vv^{>}(L\ssl_n)\iso \fS^{(n)}$, where $\fS^{(n)}$
denotes the $\BC[\vv,\vv^{-1}]$-submodule of all \textbf{integral} elements,
see~\cite[Definition 3.31]{t}. We skip the definition of the latter as it is
not presently needed.
\end{Rem}


\subsection{Extended shuffle algebra $S^{(n),\geq}$}\label{ssec extended shuffle algebra}
\

For the purpose of the next section, define the \emph{extended shuffle algebra}
(cf.~\cite[\S3.4]{n}) $S^{(n),\geq}$ by adjoining pairwise commuting generators
$\{\psi^-_{i,-s},(\psi^-_{i,0})^{-1}\}_{i\in I}^{s\in \BN}$ with the following relations:
\begin{equation}\label{cartan shuffle action}
  \psi^-_i(z)\star F =
  \left[F\left(\{x_{j,r}\}_{j\in I}^{1\leq r\leq k_j}\right)\cdot \prod_{j\in I}\prod_{r=1}^{k_j}
        \frac{\zeta_{i,j}(z/x_{j,r})}{\zeta_{j,i}(x_{j,r}/z)} \right]\star \psi^-_i(z)
\end{equation}
for any $F\in S^{(n)}_{\unl{k}}$, where we set
$\psi^-_i(z):=\sum_{s\geq 0} \psi^-_{i,-s}z^s$,
$\star$ denotes the multiplication in $S^{(n),\geq}$,
and the $\zeta$-factors in the right-hand side are all
expanded in the non-negative powers of $z$.

Then, the isomorphism $\Psi$ of Theorem~\ref{shuffle isomorphism} naturally
extends to a $\BC(\vv)$-algebra isomorphism
\begin{equation}\label{extended shuffle isomorphism}
  \Psi\colon U^\geq_\vv(L\ssl_n)\iso S^{(n),\geq}
  \ \ \mathrm{with}\ \ \psi^-_{i,-s}\mapsto \psi^-_{i,-s}.
\end{equation}
Evoking the new Drinfeld Hopf algebra structure on $U^\geq_\vv(L\ssl_n)$
of Section~\ref{ssec new Drinfeld pairing},~(\ref{extended shuffle isomorphism})
induces the one on $S^{(n),\geq}$. The corresponding coproduct $\Delta$
is given by (cf.~\cite[Proposition 3.5]{n}):
\begin{equation}\label{coproduct Cartan}
  \Delta(\psi^-_i(z))=\psi^-_i(z)\otimes \psi^-_i(z),
\end{equation}
\begin{equation}\label{coproduct shuffle}
  \Delta(F)=\sum_{\unl{\ell}\in \BN^I}^{\unl{\ell}\leq \unl{k}}
  \frac{\left[\prod_{i\in I}\prod_{r>\ell_i} \psi^-_i(x_{i,r})\right]
        \star F(x_{i,r\leq \ell_i}\otimes x_{i,s>\ell_i})}
       {\prod_{i,j\in I}\prod_{r\leq \ell_i}^{s> \ell_j} \zeta_{j,i}(x_{j,s}/x_{i,r})}
\end{equation}
for $F\in S^{(n)}_{\unl{k}}$, where
$\unl{\ell}\leq \unl{k}$ iff $\ell_i\leq k_i$ for all $i$.
We expand the right-hand side of~(\ref{coproduct shuffle}) in the
non-negative powers of $x_{j,s}/x_{i,r}$ for $s>\ell_j$ and $r\leq \ell_i$,
put the symbols $\psi^-_{i,-s}$ to the very left, then all powers
of $x_{i,r}$ with $r\leq \ell_i$, then the $\otimes$ sign,
and finally all powers of $x_{i,r}$ with $r>\ell_i$.



\section{Main result}\label{sec main result}

The main result of this note is the duality of the integral forms
$\sU_\vv(L\ssl_n)$ and $\fU_\vv(L\ssl_n)$ with respect to the
$\BC(\vv)$-valued new Drinfeld pairing $\varphi$ on $U_\vv(L\ssl_n)$
of Proposition~\ref{Drinfeld double sln}:

\begin{Thm}\label{Main Theorem}
(a)
  $\sU^>_\vv(L\ssl_n)=\{x\in U^>_\vv(L\ssl_n)|
   \varphi(x,y)\in \BC[\vv,\vv^{-1}]\ \mathrm{for\ all}\ y\in \fU^<_\vv(L\ssl_n)\}$.

\noindent
(b)
  $\sU^<_\vv(L\ssl_n)=\{y\in U^<_\vv(L\ssl_n)|
   \varphi(x,y)\in \BC[\vv,\vv^{-1}]\ \mathrm{for\ all}\ x\in \fU^>_\vv(L\ssl_n)\}$.
\end{Thm}

\begin{proof}
We shall prove only part (a) as the proof of part (b) is completely analogous.
Our proof is crucially based on the PBWD result for $\fU^<_\vv(L\ssl_n)$,
Theorem~\ref{PBWD for RTT integral form}(c1), the shuffle realization of
$\sU^>_\vv(L\ssl_n)$, Theorem~\ref{shuffle Grojnowski isomorphism}, and
the shuffle realization~(\ref{coproduct shuffle}) of the new Drinfeld coproduct.
We will first establish Theorem~\ref{Main Theorem}(a) for $n=2$, and then generalize
our arguments to $n>2$.

\medskip
\noindent
\underline{\emph{Case n=2}}.
For $n=2$, we shall skip the first index $i$.
Set $\wt{f}(z):=(\vv-\vv^{-1})f(z)=\sum_{r\in \BZ}\wt{f}_rz^{-r}$.
Then, Theorem~\ref{Main Theorem}(a) is equivalent to:
\begin{equation}\label{n=2 equivalence}
  \sU^>_\vv(L\ssl_2)=
  \left\{x\in U^>_\vv(L\ssl_2):\varphi\left(x,\wt{f}(z_1)\cdots\wt{f}(z_N)\right)\in \BC[\vv,\vv^{-1}][[z_1^{\pm 1},\ldots,z_N^{\pm 1}]] \ \forall\ N\right\}.
\end{equation}
The algebra $U_\vv(L\ssl_2)$ is $\BZ$-graded via
$\deg(e_r)=1, \deg(f_r)=-1, \deg(\psi^\pm_{\pm s})=0$
for any $r\in \BZ,s\in \BN$. In particular,
  $U^>_\vv(L\ssl_2)=\oplus_{k\in \BN} U^>_\vv(L\ssl_2)[k]$
with $U^>_\vv(L\ssl_2)[k]$ consisting of all degree $k$ elements.
Due to~(\ref{pairing formulas}), the new Drinfeld pairing $\varphi$
is \emph{of degree zero}, that is
\begin{equation}\label{n=2 degree zero pairing}
  \varphi(x,y)=0\
  \mathrm{for\ homogeneous\ elements}\ x,y\ \mathrm{with}\ \deg(x)+\deg(y)\ne 0.
\end{equation}

Since $U^>_\vv(L\ssl_2)[1]$ is spanned by $\{e_r\}_{r\in \BZ}$ and
  $\varphi(e_r,\wt{f}(z_1))=z_1^r=\Psi(e_r)_{\mid x_1\mapsto z_1}$,
we get
\begin{equation}\label{n=2 basic pairing}
  \varphi\left(x,\wt{f}(z_1)\right)=\Psi(x)_{\mid x_1\mapsto z_1}
  \ \ \mathrm{for\ any}\ \ x\in U^>_\vv(L\ssl_2)[1].
\end{equation}
Combining~(\ref{n=2 basic pairing}) with the shuffle
formulas~(\ref{coproduct Cartan},~\ref{coproduct shuffle}) for the
new Drinfeld coproduct $\Delta$ and the property~(\ref{Hopf pairing 1}),
we obtain the general formula for the pairing with $\wt{f}(z_1)\cdots\wt{f}(z_N)$:

\begin{Lem}\label{n=2 general pairing}
For $x\in U^>_\vv(L\ssl_2)[k]$, we have
\begin{equation}\label{n=2 pairing}
  \varphi\left(x,\wt{f}(z_1)\cdots\wt{f}(z_N)\right)=
  \delta_{k,N}\cdot \Psi(x)_{\mid x_r\mapsto z_r}\cdot \prod_{1\leq r<s\leq N}\zeta^{-1}(z_r/z_s)
\end{equation}
with the factors $\zeta^{-1}(z_r/z_s)$ expanded in the non-negative powers of $z_s/z_r$.
\end{Lem}

\begin{proof}
Due to~(\ref{n=2 degree zero pairing}), we have
  $\varphi\left(x,\wt{f}(z_1)\cdots\wt{f}(z_N)\right)=0$ if $k\ne N$.
Henceforth, we will assume $k=N$. Set $F:=\Psi(x)\in S^{(2)}_N$,
so that $F=F(x_1,\ldots,x_N)$ is a symmetric Laurent polynomial.

Due to the property~(\ref{Hopf pairing 1}), we have
\begin{equation}\label{pairing step 0}
  \varphi\left(x,\wt{f}(z_1)\cdots\wt{f}(z_N)\right)=
  \varphi\left(\Delta^{(N-1)}(x),\wt{f}(z_1)\otimes \cdots\otimes \wt{f}(z_N)\right),
\end{equation}
where
  $\Delta^{(\ell)}\colon U^\geq_\vv(L\ssl_n)\to {U^\geq_\vv(L\ssl_n)}^{\otimes (\ell+1)}
   \ (\ell\in \BZ_{>0})$
are defined inductively via
\begin{equation*}
  \Delta^{(1)}:=\Delta\ \mathrm{and}\
  \Delta^{(\ell)}:=(\Delta\otimes \mathrm{Id}^{\otimes(\ell-1)})\circ \Delta^{(\ell-1)}
  \ \mathrm{for}\ \ell\geq 2.
\end{equation*}
Evoking the formulas~(\ref{coproduct Cartan},~\ref{coproduct shuffle})
and the property~(\ref{n=2 degree zero pairing}), we obtain
\begin{equation}\label{pairing step 1}
  \varphi\left(\Delta^{(N-1)}(x),\wt{f}(z_1)\otimes \cdots\otimes \wt{f}(z_N)\right)=
  \varphi\left(\Psi^{-1}(G),\wt{f}(z_1)\otimes \cdots\otimes \wt{f}(z_N)\right),
\end{equation}
where
\begin{equation}\label{pairing step 2}
  G=
  \frac{(\prod_{r=2}^N \psi^-(x_r)\otimes \prod_{r=3}^N \psi^-(x_r)\otimes \cdots \otimes \psi^-(x_N)\otimes 1)
         \star F(x_1\otimes x_2\otimes\cdots \otimes x_N)}
       {\prod_{1\leq r<s\leq N} \zeta(x_s/x_r)}.
\end{equation}

Recalling the properties~(\ref{Hopf pairing 1},~\ref{n=2 degree zero pairing})
and the formula~(\ref{n=2 basic pairing}), we get
\begin{equation}\label{pairing step 3}
\begin{split}
    & \varphi\left(\psi^-(t_1)\cdots\psi^-(t_\ell)x,\wt{f}(z_1)\right)=\\
    & \prod_{r=1}^\ell \varphi\left(\psi^-(t_r),\psi^+(z_1)\right)\cdot \varphi\left(x,\wt{f}(z_1)\right)=
      \prod_{r=1}^\ell \frac{\zeta(t_r/z_1)}{\zeta(z_1/t_r)}\cdot \Psi(x)_{\mid x_1\mapsto z_1}
\end{split}
\end{equation}
for $x\in U^>_\vv(L\ssl_2)[1]$,
with the right-hand side expanded in the non-negative powers of $t_r/z_1$.
Combining~(\ref{pairing step 0})--(\ref{pairing step 3}), we finally obtain
\begin{equation}\label{pairing step 4}
  \varphi\left(x,\wt{f}(z_1)\cdots\wt{f}(z_N)\right)=
  \Psi(x)_{\mid x_r\mapsto z_r}\cdot \prod_{1\leq r<s\leq N}\zeta^{-1}(z_r/z_s).
\end{equation}
This completes our proof of Lemma~\ref{n=2 general pairing}.
\end{proof}

Thus, the $\BC[\vv,\vv^{-1}]$-submodule of $U^>_\vv(L\ssl_2)$ defined by
the right-hand side of~(\ref{n=2 equivalence}) is $\BN$-graded. Moreover,
$x\in U^>_\vv(L\ssl_2)[k]$ satisfies
  $\varphi(x,\wt{f}(z_1)\cdots\wt{f}(z_N))\in
   \BC[\vv,\vv^{-1}][[z_1^{\pm 1},\ldots,z_N^{\pm 1}]]$
for all $N$ if and only if
  $\Psi(x)\in \BC[\vv,\vv^{-1}][x_1^{\pm 1},\ldots,x_k^{\pm 1}]$.
The latter is equivalent to the inclusion $x\in \sU^>_\vv(L\ssl_2)$,
due to Theorem~\ref{shuffle Grojnowski isomorphism}
(as all specialization maps $\phi_\ast$ of~(\ref{specialization map})
are trivial for $n=2$).

This completes our proof of Theorem~\ref{Main Theorem}(a)
in the smallest rank case $n=2$.

\medskip
\noindent
\underline{\emph{Case $n>2$}}. For any $1\leq j\leq i\leq n-1$,
define the series
\begin{equation}\label{higher root generating series}
  \wt{f}_{j;i}(z_j,\ldots,z_i):=(\vv-\vv^{-1})
  [\cdots[[f_{j}(z_j),f_{j+1}(z_{j+1})]_\vv,f_{j+2}(z_{j+2})]_\vv,\cdots,f_{i}(z_i)]_\vv.
\end{equation}
Note that
  $\wt{f}_{j;i}(z_j,\ldots,z_i)\in \fU^<_\vv(L\ssl_n)[[z_j^{\pm 1},\ldots,z_i^{\pm 1}]]$,
and its coefficients encode
  $\wt{f}_{\alpha_j+\ldots+\alpha_i,\unl{r}}$ of~(\ref{integral higher roots})
for all possible decompositions $\unl{r}\in \BZ^{i-j+1}$.
For $i=j$, we shall denote $\wt{f}_{j;j}(z)$ simply by $\wt{f}_{j}(z)$, so that
  $\wt{f}_{j;i}(z_j,\ldots,z_i):=(\vv-\vv^{-1})^{j-i}
   [\cdots[[\wt{f}_{j}(z_j),\wt{f}_{j+1}(z_{j+1})]_\vv,\wt{f}_{j+2}(z_{j+2})]_\vv\cdots,\wt{f}_{i}(z_i)]_\vv$.
Similar to the $n=2$ case treated above, our primary goal is to compute
the new Drinfeld pairing with products of these $\wt{f}_{j;i}(z_j,\ldots,z_i)$.

The algebra $U_\vv(L\ssl_n)$ is $\BZ^I$-graded via
$\deg(e_{i,r})=1_i, \deg(f_{i,r})=-1_i, \deg(\psi^\pm_{i,\pm s})=\unl{0}$
for all $i\in I, r\in \BZ, s\in \BN$, where $\unl{0}=(0,\ldots,0)$ and
$1_i=(0,\ldots,1,\ldots,0)$ with $1$ placed at the $i$-th spot. In particular,
  $U^>_\vv(L\ssl_n)=\oplus_{\unl{k}\in \BN^I} U^>_\vv(L\ssl_n)[\unl{k}]$
with $U^>_\vv(L\ssl_n)[\unl{k}]$ consisting of all degree $\unl{k}$ elements.
Due to~(\ref{pairing formulas}), the new Drinfeld pairing $\varphi$ is
\emph{of degree zero}, that is
\begin{equation}\label{n>2 degree zero pairing}
  \varphi(x,y)=0\
  \mathrm{for\ homogeneous\ elements}\ x,y\ \mathrm{with}\ \deg(x)+\deg(y)\ne \unl{0}.
\end{equation}

Similar to~(\ref{n=2 basic pairing}), we obtain
\begin{equation}\label{n>2 basic pairing}
  \varphi\left(x,\wt{f}_{j}(z_j)\right)=\Psi(x)_{\mid x_{j,1}\mapsto z_j}
  \ \ \mathrm{for\ any}\ \ x\in U^>_\vv(L\ssl_n)[1_j].
\end{equation}
The following result generalizes~(\ref{n>2 basic pairing}) and is proved
completely analogously to Lemma~\ref{n=2 general pairing}:

\begin{Lem}\label{n>2 general pairing 1}
For $x\in U^>_\vv(L\ssl_n)[\unl{k}]$ and any collection $j_1,\ldots,j_N\in I$,
we have
\begin{equation*}\label{n>2 pairing}
  \varphi\left(x,\wt{f}_{j_1}(z^{(1)}_{j_1})\cdots \wt{f}_{j_N}(z^{(N)}_{j_N})\right)=
  \delta_{\unl{k},1_{j_1}+\ldots+1_{j_N}}\cdot
  \Psi(x)_{\mid x_{i,r}\mapsto z^{(\ast)}_{j_\ast}}\cdot
  \prod_{1\leq r<s\leq N}\zeta^{-1}_{j_r,j_s} \left(z^{(r)}_{j_r}/z^{(s)}_{j_s}\right)
\end{equation*}
with the factors $\zeta^{-1}_{j_r,j_s}(z^{(r)}_{j_r}/z^{(s)}_{j_s})$
expanded in the non-negative powers of $z^{(s)}_{j_s}/z^{(r)}_{j_r}$.
\end{Lem}

\begin{Rem}\label{explanation 1}
The specialization $\Psi(x)_{\mid x_{i,r}\mapsto z^{(\ast)}_{j_\ast}}$ in
Lemma~\ref{n>2 general pairing 1} should be understood as follows.
For each $i\in I$, there are $k_i$ variables $\{x_{i,r}\}_{r=1}^{k_i}$
(``of color $i$'') featuring in $\Psi(x)$. Since $k_i=\#\{1\leq t\leq N|j_t=i\}$,
say $j_{t_{i,1}}=\ldots=j_{t_{i,k_i}}=i$, then we specialize
$x_{i,r}\mapsto z^{(t_{i,r})}_{j_{t_{i,r}}}$ in $\Psi(x)$.
\end{Rem}

In what follows, we use the convention that
\begin{equation}\label{distribution}
  \frac{1}{z-w} \ \mathrm{represents\ the\ series}\
  \sum_{m=0}^\infty z^{-m-1}w^m.
\end{equation}

For $1\leq j\leq i\leq n-1$, consider a graph $Q_{j,i}$ whose vertices
are labeled by $j,j+1,\ldots,i$ and the vertices $k,k+1$ ($j\leq k<i$)
are connected by a single edge. Let $\Or_{j,i}$ denote the set of all
orientations $\pi$ of $Q_{j,i}$. Evoking~(\ref{distribution}), for
$\pi\in \Or_{j,i}$ and $j\leq k<i$, define $\zeta^{-1}_{\pi,k}(z,w)$ via
\begin{equation}\label{ordered zeta}
  \zeta^{-1}_{\pi,k}(z,w):=
   \begin{cases}
     (z-w)\cdot\frac{1}{z-\vv w}, & \text{if } k\rightarrow k+1\ \text{in } \pi \\
     \vv(z-w)\cdot\frac{1}{w-\vv z}, & \text{if } k\leftarrow k+1\ \text{in } \pi
   \end{cases}.
\end{equation}

Simplifying all $[a,b]_\vv$ as $ab-\vv ba$ in~(\ref{higher root generating series}),
thus expressing the latter as a sum of $2^{i-j}$ terms,
Lemma~\ref{n>2 general pairing 1} implies the formula for the
new Drinfeld pairing with $\wt{f}_{j;i}(z_j,\ldots,z_i)$:

\begin{Lem}\label{n>2 general pairing 2}
For $x\in U^>_\vv(L\ssl_n)[\unl{k}]$ and $1\leq j\leq i<n$,
we have
\begin{equation}\label{n>2 root pairing}
  \varphi\left(x,\wt{f}_{j;i}(z_j,\ldots,z_i)\right)=
  \frac{\delta_{\unl{k},1_j+\ldots+1_i}}{(\vv-\vv^{-1})^{i-j}}
  \cdot \Psi(x)_{\mid x_{k,1}\mapsto z_{k}} \cdot
  \sum_{\pi\in \Or_{j,i}} \prod_{j\leq k<i} \zeta^{-1}_{\pi,k}(z_k,z_{k+1}).
\end{equation}
\end{Lem}

\begin{Rem}
The denominator of $\Psi(x)_{\mid x_{k,1}\mapsto z_{k}}$
is canceled by the numerators of $\zeta^{-1}_{\ast,\ast}$-factors.
\end{Rem}

\begin{Cor}\label{pairing with higher roots}
If
  $\varphi\left(x,\wt{f}_{j;i}(z_j,\ldots,z_i)\right)\in
   \BC[\vv,\vv^{-1}][[z_j^{\pm 1},\ldots,z_i^{\pm 1}]]$,
then $\phi_{\unl{d}}(\Psi(x))$ is divisible by $(\vv-\vv^{-1})^{i-j}$,
where $\phi_{\unl{d}}$ is the specialization map~(\ref{specialization map})
with $\unl{d}=\{d_\beta\}, d_\beta=\delta_{\beta,\alpha_j+\ldots+\alpha_i}$.
\end{Cor}

\begin{proof}
Due to~(\ref{n>2 degree zero pairing}), we may assume that
$x\in U^>_\vv(L\ssl_n)[1_j+\ldots+1_i]$, so that
\begin{equation}\label{x-to-p}
  \Psi(x)=\frac{p(x_{j,1},\ldots,x_{i,1})}{(x_{j,1}-x_{j+1,1})\cdots (x_{i-1,1}-x_{i,1})}
  \ \mathrm{with}\ p\in \BC(v)[x_{j,1}^{\pm 1},\ldots,x_{i,1}^{\pm 1}].
\end{equation}

First, let us assume that
  $p(x_{j,1},\ldots,x_{i,1})=x_{j,1}^{a_j}\cdots x_{i,1}^{a_i}$.
Pick sufficiently small integers $r_{j+1},\ldots,r_i\ll 0$,
so that $a_k+r_k < 0$ for $j<k\leq i$.
Then, evaluating the coefficient of $\prod_{k=j+1}^i z_k^{-r_k}$ in
the right-hand side of~(\ref{n>2 root pairing}), we get a nonzero
contribution only from $\pi\in \Or_{j,i}$ with $k\rightarrow k+1$
for all $j\leq k<i$. Moreover, the corresponding contribution equals
\begin{equation}\label{key specialization monomial}
  (\vv-\vv^{-1})^{j-i}\vv^{\sA} z_j^{\sB}\cdot \left(z_j^{a_j}(\vv^{-1}z_j)^{a_{j+1}}\cdots (\vv^{j-i}z_j)^{a_i}\right)
\end{equation}
with
\begin{equation}\label{AB powers}
  \sA=\sum_{j<k\leq i} (j-k)(r_k-1+\delta_{k,i}),\
  \sB=\sum_{j<k\leq i} (r_k-1).
\end{equation}

Note that $\sA,\sB$ of~(\ref{AB powers}) are actually independent of $a_j,\ldots,a_i$.
Thus, for any $x$ as above and the associated Laurent polynomial $p$ of~(\ref{x-to-p}),
comparing the coefficients of $\prod_{k=j+1}^i z_k^{-r_k}$ in~(\ref{n>2 root pairing})
for sufficiently small $r_{j+1},\ldots,r_i\ll 0$, we obtain
\begin{equation}\label{key specialization}
\begin{split}
  & \varphi\left(x,(\vv-\vv^{-1})[\cdots[f_{j}(z),f_{j+1,r_{j+1}}]_\vv,,\cdots,f_{i,r_i}]_\vv\right)=\\
  & (\vv-\vv^{-1})^{j-i}\cdot \vv^{\sA} z^{\sB} \cdot p(z,\vv^{-1}z,\ldots,\vv^{j-i}z).
\end{split}
\end{equation}
Combining~(\ref{key specialization}) and the definition of $\phi_{\unl{d}}$ with
  $\unl{d}=\{d_\beta\}, d_\beta=\delta_{\beta,\alpha_j+\ldots+\alpha_i}$,
~(\ref{specialization map}), we see that $\phi_{\unl{d}}(\Psi(x))$ is
indeed divisible by $(\vv-\vv^{-1})^{i-j}$. This completes our proof of
Corollary~\ref{pairing with higher roots}.
\end{proof}

Combining~(\ref{n>2 root pairing}) with the shuffle
formulas~(\ref{coproduct Cartan},~\ref{coproduct shuffle}) for the
new Drinfeld coproduct $\Delta$ and the property~(\ref{Hopf pairing 1}),
we obtain the formula for the pairing with
$\prod_{r=1}^N \wt{f}_{j_r;i_r}(z^{(r)}_{j_r},\ldots,z^{(r)}_{i_r})$:

\begin{Lem}\label{n>2 general pairing 3}
For $x\in U^>_\vv(L\ssl_n)[\sum_{r=1}^N\sum_{k=j_r}^{i_r} 1_k]$, we have
\begin{equation}\label{n>2 pairing general}
\begin{split}
  & \varphi\left(x,\wt{f}_{j_1;i_1}(z^{(1)}_{j_1},\ldots,z^{(1)}_{i_1})\cdots\wt{f}_{j_N;i_N}(z^{(N)}_{j_N},\ldots,z^{(N)}_{i_N})\right)=
    \prod_{r<s}\prod_{j_r\leq k\leq i_r}^{j_s\leq \ell\leq i_s} \zeta^{-1}_{k,\ell}(z^{(r)}_{k}/z^{(s)}_{\ell})\times\\
  & (\vv-\vv^{-1})^{\sum_{r=1}^N (j_r-i_r)}\cdot \Psi(x)_{\mid x_{i,r}\mapsto z^{(\ast)}_{i}}\cdot
    \prod_{r=1}^N \left(\sum_{\pi_r\in \Or_{j_r,i_r}} \prod_{j_r\leq k<i_r} \zeta^{-1}_{\pi_r,k}(z^{(r)}_{k},z^{(r)}_{k+1})\right)
\end{split}
\end{equation}
with the factors $\zeta^{-1}_{k,\ell}(z^{(r)}_{k}/z^{(s)}_{\ell})$ expanded
in the non-negative powers of $z^{(s)}_{\ell}/z^{(r)}_{k}$.
\end{Lem}

\begin{Rem}\label{explanation 2}
The specialization $\Psi(x)_{\mid x_{i,r}\mapsto z^{(\ast)}_{i}}$
in~(\ref{n>2 pairing general}) should be understood as follows.
For each $i\in I$, there are $k_i=\#\{1\leq t\leq N|j_t\leq i\leq i_t\}$
variables $\{x_{i,r}\}_{r=1}^{k_i}$ (``of color $i$'') featuring in $\Psi(x)$.
If $1\leq t_{i,1}<\ldots<t_{i,k_i}\leq N$ denote the corresponding indices,
such that $j_{t_{i,r}}\leq i\leq i_{t_{i,r}}$, then we specialize
$x_{i,r}\mapsto z^{(t_{i,r})}_{i}$ in $\Psi(x)$, cf.~Remark~\ref{explanation 1}.
\end{Rem}

Since the proof of Lemma~\ref{n>2 general pairing 3} is entirely analogous
to that of Lemma~\ref{n=2 general pairing}, we leave details to the interested reader.
Similar to Corollary~\ref{pairing with higher roots}, we obtain the following result:

\begin{Cor}\label{pairing with products of higher roots}
If
  $\varphi\left(x,\wt{f}_{j_1;i_1}(z^{(1)}_{j_1},\ldots,z^{(1)}_{i_1})\cdots\wt{f}_{j_N;i_N}(z^{(N)}_{j_N},\ldots,z^{(N)}_{i_N})\right)$
is a $\BC[\vv,\vv^{-1}]$-valued Laurent polynomial in
$\{z^{(r)}_i\}_{1\leq r\leq N}^{j_r\leq i\leq i_r}$, then
$\phi_{\unl{d}}(\Psi(x))$ is divisible by $(\vv-\vv^{-1})^{\sum_{r=1}^N (i_r-j_r)}$,
where $\phi_{\unl{d}}$ is the specialization map~(\ref{specialization map}) with
$\unl{d}=\{d_\beta\},\ d_{\alpha_j+\ldots+\alpha_i}=\#\{1\leq r\leq N|j_r=j, i_r=i\}$.
\end{Cor}

This result, combined with Theorem~\ref{shuffle Grojnowski isomorphism},
implies the inclusion ``$\supseteq$'' in Theorem~\ref{Main Theorem}(a):

\begin{Prop}\label{inclusion one way}
  $\sU^>_\vv(L\ssl_n)\supseteq\{x\in U^>_\vv(L\ssl_n)| \varphi(x,y)\in \BC[\vv,\vv^{-1}]\ \mathrm{for\ all}\ y\in \fU^<_\vv(L\ssl_n)\}$.
\end{Prop}

Thus, it remains to establish the opposite inclusion ``$\subseteq$'' in Theorem~\ref{Main Theorem}(a):

\begin{Prop}\label{inclusion other way}
  $\sU^>_\vv(L\ssl_n)\subseteq\{x\in U^>_\vv(L\ssl_n)| \varphi(x,y)\in \BC[\vv,\vv^{-1}]\ \mathrm{for\ all}\ y\in \fU^<_\vv(L\ssl_n)\}$.
\end{Prop}

\begin{proof}
Our proof will proceed in several steps by reducing to the setup
in which~(\ref{key specialization}) applies.


First, evoking the shuffle realization of the subalgebra $\sU^>_\vv(L\ssl_n)$,
Theorem~\ref{shuffle Grojnowski isomorphism}, and of the new Drinfeld coproduct,
formula~(\ref{coproduct shuffle}), we immediately obtain the following result:

\begin{Lem}\label{Grojnowski coproduct closed}
For any $x\in \sU^>_\vv(L\ssl_n)$, we have $\Delta(x)=x'_{(1)}x_{(1)}\otimes x_{(2)}$
in the Sweedler notation (the right-hand side is an infinite sum) with
$x_{(1)},x_{(2)}\in \sU^>_\vv(L\ssl_n)$ and $x'_{(1)}$--a monomial in $\psi^-_{\ast,\ast}$.
\end{Lem}


Combining Lemma~\ref{Grojnowski coproduct closed} with the property~(\ref{Hopf pairing 1}),
it thus suffices to show that given any $x\in \sU^>_\vv(L\ssl_n)[1_j+\ldots+1_i]$,
$x'$--a monomial in $\psi^-_{\ast,\ast}$, and $\unl{r}=(r_j,\ldots,r_i)\in \BZ^{i-j+1}$,
we have
\begin{equation}\label{reduction 1}
  \varphi\left(x'x,\wt{f}_{\alpha_j+\ldots+\alpha_i,\unl{r}}\right)\in \BC[\vv,\vv^{-1}].
\end{equation}
Evoking the property~(\ref{Hopf pairing 1}) once again,
for the proof of~(\ref{reduction 1}) it suffices to establish
\begin{equation}\label{reduction 2}
  \varphi\left(x,\wt{f}_{\alpha_j+\ldots+\alpha_i,\unl{r}}\right)\in \BC[\vv,\vv^{-1}]
\end{equation}
for any $x\in \sU^>_\vv(L\ssl_n)[1_j+\ldots+1_i]$ and
any $\unl{r}=(r_j,\ldots,r_i)\in \BZ^{i-j+1}$.

We shall prove~(\ref{reduction 2}) by induction in $i-j$.
The base case $i=j$ is obvious. Given $x\in \sU^>_\vv(L\ssl_n)[1_j+\ldots+1_i]$,
the validity of~(\ref{reduction 2}) for $\unl{r}=(r_j,\ldots,r_i)$ with
sufficiently small $r_{j+1},\ldots,r_i\ll 0$ is due to~(\ref{key specialization}).
We shall call such $\unl{r}\in \BZ^{i-j+1}$ ``$x$--sufficiently small''.
To establish~(\ref{reduction 2}) for a general $\unl{r}$, we shall apply
the PBWD result of Theorem~\ref{PBWD for RTT integral form}(c1) with the choice
of decompositions $\unl{r}(\beta,r)$ such that $\unl{r}(\alpha_j+\ldots+\alpha_i,r)$
are all ``$x$--sufficiently small''.
Then, combining Theorem~\ref{PBWD for RTT integral form}(c1)
with the $\BZ^I$-grading on $U^<_\vv(L\ssl_n)$, we see that the element
$\wt{f}_{\alpha_j+\ldots+\alpha_i,\unl{r}}$ can be written as a
$\BC[\vv,\vv^{-1}]$--linear combination of
  $\wt{f}_{\alpha_j+\ldots+\alpha_i,\unl{r}(\alpha_j+\ldots+\alpha_i,r)}\ (r\in \BZ)$
and degree $>1$ ordered monomials in $\wt{f}_{\alpha_{j'}+\ldots+\alpha_{i'},\unl{r}'}$
with $j\leq j'\leq i'\leq i$ and $i'-j'<i-j$.
By the above observation,
  $\varphi\left(x,\wt{f}_{\alpha_j+\ldots+\alpha_i,\unl{r}(\alpha_j+\ldots+\alpha_i,r)}\right)\in \BC[\vv,\vv^{-1}]$
for any $r\in \BZ$. Finally, we claim that the pairing of $x$ with degree $>1$ monomials in
$\wt{f}_{\alpha_{j'}+\ldots+\alpha_{i'},\unl{r}'}$ is $\BC[\vv,\vv^{-1}]$-valued.
To see this, apply the above arguments ((\ref{Hopf pairing 1})
and Lemma~\ref{Grojnowski coproduct closed}) again, subsequently reducing
to~(\ref{reduction 2}) with $(j,i)$ replaced by $(j',i')$,
which is established by the induction assumption.

This completes our proof of Proposition~\ref{inclusion other way}.
\end{proof}

Combining Propositions~\ref{inclusion one way},~\ref{inclusion other way},
we get the proof of Theorem~\ref{Main Theorem}(a) for arbitrary $n$.
\end{proof}

\begin{Rem}\label{importance of all decompositions}
The above proof of Theorem~\ref{Main Theorem} is crucially based on our
construction of the entire family of Poincar\'{e}-Birkhoff-Witt-Drinfeld
bases of $\fU_\vv(L\ssl_n)$ for all decompositions $\unl{r}$
(rather than picking the canonical one $\unl{r}^{(0)}$ of~(\ref{simplest decomposition})).
\end{Rem}

\begin{Rem}\label{finite counterpart of DCL}
The finite counterpart of Theorem~\ref{Main Theorem}, where $U_\vv(L\ssl_n)$ is
replaced with $U_\vv(\ssl_n)$ and the new Drinfeld pairing $\varphi$ is replaced
with the Drinfeld-Jimbo pairing, is well-known, see e.g.~\cite[\S3]{dcl}.
In the~\emph{loc.cit.}, this duality is extended to the duality between the
Cartan-extended subalgebras $\sU^{',\geq}(\ssl_n)$ and $\fU^{',\leq}(\ssl_n)$
(resp.\  $\sU^{',\leq}(\ssl_n)$ and $\fU^{',\geq}(\ssl_n)$)), where $'$ is used
to indicate yet enlarged algebras by adding more Cartan elements, see~\cite[Theorem 3.1]{dcl}.
\end{Rem}

\begin{Rem}\label{Opposite duality}
Combining Theorem~\ref{Main Theorem} with~\cite{n},
one can easily derive the opposite dualities:
\begin{equation}\label{opposite 1}
  \fU^>_\vv(L\ssl_n)=\{x\in U^>_\vv(L\ssl_n)|
  \varphi(x,y)\in \BC[\vv,\vv^{-1}]\ \mathrm{for\ all}\ y\in \sU^<_\vv(L\ssl_n)\},
\end{equation}
\begin{equation}\label{opposite 2}
  \fU^<_\vv(L\ssl_n)=\{y\in U^<_\vv(L\ssl_n)|
  \varphi(x,y)\in \BC[\vv,\vv^{-1}]\ \mathrm{for\ all}\ x\in \sU^>_\vv(L\ssl_n)\}.
\end{equation}
Viewing $U_\vv(L\ssl_n)$ as the ``vertical'' subalgebra of the
quantum toroidal algebra $U_{\vv,\bar{\vv}}(\ddot{\gl}_n)$, the results
of~\cite[Lemma 3.14, \S3.34]{n} imply that $\preceq$-ordered PBWD monomials
in $\wt{E}_{\beta,r}$ (resp.\ $\mathsf{E}^{(\ast)}_{\beta,r}$) are dual (up to $(-1)^\ast\vv^\ast$)
to $\preceq^{op}$-ordered PBWD monomials in $\mathsf{F}^{(\ast)}_{\beta,r}$
(resp.\ $\wt{F}_{\beta,r}$) with respect to the new Drinfeld
pairing $\varphi$ of Proposition~\ref{Drinfeld double sln}.
Here, we use the following notations:
\begin{itemize}
\item
  the elements
    $\{\wt{E}_{\alpha_j+\alpha_{j+1}+\ldots+\alpha_i,r},
     \wt{F}_{\alpha_j+\alpha_{j+1}+\ldots+\alpha_i,r}\}_{1\leq j\leq i<n}^{r\in \BZ}$
  are defined via~(\ref{integral higher roots}) for a specific choice
  of the decomposition $\unl{r}=(r_j,r_{j+1},\ldots,r_i)$ with
    $r_k:=\lfloor \frac{r(k-j+1)}{i-j+1} \rfloor - \lfloor \frac{r(k-j)}{i-j+1} \rfloor$,
  cf.~\cite[(3.46, 3.47)]{n} where $\frac{r}{i-j+1}$ is referred to as the ``slope'' of these elements;
\item
  the elements
    $\{\mathsf{E}^{(k)}_{\beta,r},\mathsf{F}^{(k)}_{\beta,r}\}_{\beta\in \Delta^+}^{r\in \BZ,k\in \BN}$
  are the $k$-th divided powers defined via (cf.~(\ref{devided higher root power}))
\begin{equation*}
    \mathsf{E}^{(k)}_{\beta,r}:=
    (\wt{E}_{\beta,r}/(\vv-\vv^{-1}))^{k}/[k]_\vv!
    \ \ \mathrm{and}\ \ 
    \mathsf{F}^{(k)}_{\beta,r}:=
    (\wt{F}_{\beta,r}/(\vv-\vv^{-1}))^{k}/[k]_\vv!
\end{equation*}
\item
  $\preceq^{op}$ is the opposite of the total ordering $\preceq$ on $\Delta^+\times \BZ$,
  the latter being defined via: $(\alpha_j+\ldots+\alpha_i,r)\preceq (\alpha_{j'}+\ldots+\alpha_{i'},r')$
  iff $\frac{r}{i-j+1}<\frac{r'}{i'-j'+1}$ or $\frac{r}{i-j+1}=\frac{r'}{i'-j'+1}$
  and either $i-j<i'-j'$ or $i-j=i'-j'$ and $i\leq i'$.
\end{itemize}
As the new Drinfeld pairing $\varphi$ is non-degenerate and
  $\mathsf{E}^{(k)}_{\beta,r}\in \sU^>_\vv(L\ssl_n),
   \mathsf{F}^{(k)}_{\beta,r}\in \sU^<_\vv(L\ssl_n)$
(note that $\mathsf{E}^{(k)}_{\beta,r}$ is the image of $\se_{\beta,r}^{(k)}$ under
an automorphism of $\sU^>_\vv(L\ssl_n),\ e_{k,t}\mapsto e_{k,t+r_k}$),
we obtain the inclusions ``$\supseteq$'' in~(\ref{opposite 1},~\ref{opposite 2}),
while the opposite inclusions ``$\subseteq$''\textbf{} are implied by Theorem~\ref{Main Theorem}.

These arguments also imply that the above ordered monomials
in $\wt{E}_{\beta,r}, \wt{F}_{\beta,r}, \mathsf{E}^{(\ast)}_{\beta,r}, \mathsf{F}^{(\ast)}_{\beta,r}$
form bases of the free $\BC[\vv,\vv^{-1}]$-modules
$\fU^>_\vv(L\ssl_n),\fU^<_\vv(L\ssl_n),\sU^>_\vv(L\ssl_n),\sU^<_\vv(L\ssl_n)$, respectively.
\end{Rem}



\begin{thebibliography}{99}

\bibitem[CP]{cp}
V.~Chari, A.~Pressley,
  {\em Quantum affine algebras at roots of unity},
Represent.\ Theory {\bf 1} (1997), 280--328.

\bibitem[D]{d1}
V.~Drinfeld,
  {\em A New realization of Yangians and quantized affine algebras},
Sov.\ Math.\ Dokl.\  {\bf 36} (1988), no.~2, 212--216.

\bibitem[DCL]{dcl}
C.~De Concini, V.~Lyubashenko,
  {\em Quantum function algebra at roots of $1$},
Adv.\ Math.\ {\bf 108} (1994), no.~2, 205--262.

\bibitem[DF]{df}
J.~Ding, I.~Frenkel,
  {\em Isomorphism of two realizations of quantum affine algebra $U_q(\widehat{\gl(n)})$},
Comm.\ Math.\ Phys.\ {\bf 156} (1993), no.~2, 277--300.

\bibitem[DI]{di}
J.~Ding, K.~Iohara,
  {\em Generalization of Drinfeld quantum affine algebras},
Lett.\ Math.\ Phys.\ {\bf 41} (1997), no.~2, 181--193.

\bibitem[FRT]{frt}
L.~Faddeev, N.~Reshetikhin, L.~Takhtadzhyan,
  {\em Quantization of Lie groups and Lie algebras},
(Russian) Algebra i Analiz {\bf 1} (1989), no.~1, 178--206;
translation in Leningrad Math.\ J.\ {\bf 1} (1990), no.~1, 193--225.

\bibitem[FT]{ft}
M.~Finkelberg, A.~Tsymbaliuk,
  {\em Shifted quantum affine algebras: integral forms in type $A$}
(with appendices by A.~Tsymbaliuk, A.~Weekes),
Arnold Math.\ J.\ {\bf 5} (2019), no.~2-3, 197--283.

\bibitem[G]{g}
I.~Grojnowski,
  {\em Affinizing quantum algebras: from $D$-modules to $K$-theory},
unpublished manuscript of November 11, 1994, available at
https://www.dpmms.cam.ac.uk/$\sim$groj/char.ps.

\bibitem[H]{h}
D.~Hernandez,
  {\em Representations of quantum affinizations and fusion product},
Transform.\ Groups {\bf 10} (2005), no.~2, 163--200.

\bibitem[KRT]{krt}
C.~Kassel, M.~Rosso, V.~Turaev,
  {\em Quantum groups and knot invariants},
Panoramas et Synth\`{e}ses, 5.\ Soci$\mathrm{\acute{e}}$t$\mathrm{\acute{e}}$ Math$\mathrm{\acute{e}}$matique de France,
Paris, 1997.

\bibitem[L]{l}
G.~Lusztig,
  {\em Quantum groups at roots of $1$},
Geom.\ Dedicata {\bf 35} (1990), no.~1-3, 89--113.

\bibitem[N]{n}
A.~Negut,
  {\em Quantum toroidal and shuffle algebras},
preprint, arXiv:1302.6202v4.

\bibitem[T]{t}
A.~Tsymbaliuk,
  {\em PBWD bases and shuffle algebra realizations for
       $U_\vv(L\ssl_n),U_{\vv_1,\vv_2}(L\ssl_n),U_\vv(L\ssl(m|n))$ and their integral forms},
preprint, arXiv:1808.09536v3.

\end{thebibliography}
\end{document}